\documentclass{amsart}
\usepackage{amsmath,amsthm,amssymb,amscd,tikz-cd}
\theoremstyle{plain}

\newtheorem{thm}{Theorem}[section]
\newtheorem{prop}[thm]{Proposition}
\newtheorem{lemma}[thm]{Lemma}
\newtheorem{cor}[thm]{Corollary}

\theoremstyle{definition}

\newtheorem{defn}[thm]{Definition}
\newtheorem{remark}[thm]{Remark}
\newtheorem{conj}[thm]{Conjecture}

\newcommand{\D}{{\mathcal D}}

\newcommand{\E}{{\mathcal E}}
\newcommand{\F}{{\mathbb F}}
\renewcommand{\H}{{\mathcal H}}
\newcommand{\I}{{\mathcal I}}

\renewcommand{\O}{{\mathcal O}}
\newcommand{\m}{{{\mathfrak m}}}
\newcommand{\om}{\overline{\mathfrak m}}
\newcommand{\omk}{\overline{\mathfrak m}_k}

\newcommand{\MM}{\mathcal M}

\newcommand{\MMv}[1]{\MM_{#1}}

\newcommand{\T}{{{\mathbb T}}}

\newcommand{\W}{{\mathcal W}}
\newcommand{\ve}{\varepsilon}

\newcommand{\p}{{\mathfrak p}}

\newcommand{\Z}{{\mathbb Z}}
\newcommand{\Q}{{\mathbb Q}}
\newcommand{\Qinf}{{\mathbb Q}_\infty}
\newcommand{\Qcyc}{{\mathbb Q(\mu_{p^\infty})}}
\newcommand{\R}{{\mathbb R}}
\newcommand{\C}{{\mathbb C}}
\newcommand{\Fp}{{\mathbb F}_p}

\newcommand{\Qpbar}{\overline{{\Q}}_p}

\newcommand{\Qbar}{\overline{{\Q}}}

\renewcommand{\P}{{\mathbb P}}
\newcommand{\cP}{{\mathcal P}}
\newcommand{\Zp}{\Z_p}
\newcommand{\Zpx}{\Z_p^\times}
\newcommand{\Cpx}{\C_p^\times}
\newcommand{\Qp}{\Q_p}
\newcommand{\Cp}{\C_p}

\newcommand{\inj}{\hookrightarrow}
\newcommand{\surj}{\twoheadrightarrow}

\newcommand{\lra}{\longrightarrow}

\newcommand{\smallmat}{\bigl( \begin{smallmatrix} a & b \\ c & d
\end{smallmatrix} \bigr)}

\newcommand{\sigop}{\Sigma_0(p)}

\newcommand{\psmallmat}[4]{\left( \begin{smallmatrix} #1 & #2 \\ #3 & #4 \end{smallmatrix} \right)}

\newcommand{\X}{\Zpx \times \Zp}
\newcommand{\M}{\Meas(\Zp)}
\newcommand{\Mx}{\Meas(\Zpx)}

\newcommand{\Mk}{\Meas_k(\Zp)}
\newcommand{\Mv}[1]{\Meas_{#1}(\Zp)}

\newcommand{\Mvx}[1]{\Meas_{#1}(\Zpx)}
\newcommand{\Ms}{\Meas(\Zpx \times \Zp)}
\newcommand{\Mg}{\Zp[[\Zp]]}
\newcommand{\Mxg}{\Zp[[\Zpx]]}
\newcommand{\Msg}{\Zp[[\Zpx \times \Zp]]}
\newcommand{\Lt}{\tilde{\Lambda}}

\newcommand{\rhobar}{\overline{\rho}}

\newcommand{\Kf}{K}
\newcommand{\Of}{\O}
\newcommand{\KOf}{\Kf/\Of}
\newcommand{\Ff}{\F}
\newcommand{\Afj}{A_{f,j}}

\newcommand{\Gammac}{\Zpx}

\newcommand{\Gammawj}{\Gamma_{w,j}}
\newcommand{\Leis}{L_{\eis}}
\newcommand{\Leispsi}{L_{\eis,\m}}
\newcommand{\Tk}{\T_k}
\newcommand{\Tc}{\T^c}
\newcommand{\Tck}{\T^c_k}

\newcommand{\mc}{\m^c}
\newcommand{\mk}{\m_k}
\newcommand{\mck}{\m^c_k}

\newcommand{\Tm}{\T_{\m}}
\newcommand{\Tcm}{\Tc_{\mc}}
\newcommand{\Tkm}{\T_{k,\m_k}}
\newcommand{\Tckm}{\Tc_{k,\mck}}

\newcommand{\Ieis}{\I}
\newcommand{\Ikeis}{\I_k}
\newcommand{\Iceis}{\I^c}
\newcommand{\Ickeis}{\I^c_k}
\newcommand{\vandv}[1]{\text{\rm (Vand~}\omega^{#1}{\rm )}}
\newcommand{\vandpsiv}[2]{\text{\rm (Vand~}#1\omega^{#2}{\rm )}}
\newcommand{\lu}{\text{\rm ($U_p-1$~gens)}}
\newcommand{\fk}{\text{\rm(Good Eisen)}}
\newcommand{\all}{\text{\rm($\star$)}}
\newcommand{\ep}{\text{Eisen}$^+$}

\DeclareMathOperator{\Div}{Div}
\DeclareMathOperator{\Hom}{Hom}
\DeclareMathOperator{\Ext}{Ext}
\DeclareMathOperator{\ord}{ord}

\DeclareMathOperator{\im}{im}

\DeclareMathOperator{\Sym}{Sym}

\DeclareMathOperator{\GL}{GL}
\DeclareMathOperator{\SL}{SL}

\DeclareMathOperator{\Cl}{Cl}
\DeclareMathOperator{\Sel}{Sel}

\DeclareMathOperator{\Aut}{Aut}
\DeclareMathOperator{\Frob}{Frob}
\DeclareMathOperator{\eis}{eis}

\DeclareMathOperator{\Gal}{Gal}
\DeclareMathOperator{\chr}{char}
\DeclareMathOperator{\Meas}{Meas}
\DeclareMathOperator{\Cont}{Cont}

\DeclareMathOperator{\Ind}{Ind}
\DeclareMathOperator{\res}{res}
\DeclareMathOperator{\un}{un}
\DeclareMathOperator{\loc}{loc}
\DeclareMathOperator{\cont}{cont}

\renewcommand{\and}{\quad \text{and} \quad}

\title{Congruences with Eisenstein series and $\mu$-invariants}

\author{Jo\"el Bella\"iche and Robert Pollack}

\begin{document}

\maketitle

\begin{abstract}
We study the variation of $\mu$-invariants in Hida families with residually reducible Galois representations.  We prove a lower bound for these invariants which is often expressible in terms of the $p$-adic zeta function.  This lower bound forces these $\mu$-invariants to be unbounded along the family, and moreover, we conjecture that this lower bound is an equality.  When $U_p-1$ generates the cuspidal Eisenstein ideal, we establish this conjecture and further prove that the $p$-adic $L$-function is simply a power of $p$ up to a unit (i.e.\ $\lambda=0$).  On the algebraic side, we prove analogous statements for the associated Selmer groups which, in particular, establishes the main conjecture for such forms.  

\end{abstract}

\section{Introduction}

\subsection{Setting the stage}
Let $(p,k)$ be an irregular pair -- that is, $p$ is an irregular prime such that $p$ divides the numerator of $B_k$, the $k$-th Bernoulli number.  In \cite{Ribet-Herbrand}, Ribet proved the converse of Herbrand's theorem which predicts the non-triviality of a particular eigenspace of the $p$-part of the class group of $\Q(\mu_p)$ under the action of $\Gal(\Q(\mu_p)/\Q)$.  His method exploited a congruence between an Eisenstein series and a cuspform.   Ribet worked in weight 2 at level $p$, but we describe the idea here in weight $k$ and level $1$; namely, let
$$
E_k = -\frac{B_k}{2k} + \sum_{n \geq 1} \sigma_{k-1}(n) q^n
$$
denote the Eisenstein series of weight $k >2$ on $\SL_2(\Z)$.  Under the assumption that $(p,k)$ is an irregular pair, the constant term of $E_k$ vanishes mod $p$, and thus $E_k$ ``looks like" a cuspform modulo $p$.  One can make this idea precise and prove the existence of a cuspidal eigenform $g_k$ of weight $k$ and level 1 which is congruent to the Eisenstein series $E_k$.  The mod $p$ Galois representation of $g_k$ is then reducible, and from this Galois representation one can extract the sought after unramified extension of $\Q(\mu_p)$. 

Wiles  in \cite{Wiles-MC} pushed this argument further by looking at the whole Hida family of the $g_k$ as $k$ varies $p$-adically.  By analyzing the intersection of the Eisenstein and cuspidal branches of this Hida family (in the Hilbert modular case), Wiles proved the main conjecture over totally real fields.

In this paper, rather than looking at the Iwasawa theory of class groups, we will instead focus on the Iwasawa theory of these famous cuspforms $g_k$ in their own right.  Namely, we will examine the $p$-adic $L$-functions and Selmer groups of the $g_k$ as $k$ varies $p$-adically.  In fact, within the paper, we will study a larger collection of cuspforms with reducible residual representations by allowing for congruences with Eisenstein series with characters.  However, to keep things simple, for the remainder of the introduction we will continue to work with tame level $N=1$.

To further set the stage, let's recall the Hida theoretic picture of congruences between Eisenstein series and cuspforms.  Namely, to put the $E_k$ in a $p$-adic family, we must consider $E_k^{\ord} := E_k(z) - p^{k-1} E_k(pz)$, the ordinary $p$-stabilization to $\Gamma_0(p)$.  Explicitly, we have
$$
E_k^{\ord} = -(1-p^{k-1}) \frac{B_k}{2k} + \sum_{n \geq 1} \sigma^{(p)}_{k-1}(n) q^n
$$
where $\sigma^{(p)}_{k-1}(n)$ is the sum of the $(k-1)$-st powers of the prime-to-$p$ divisors of $n$.  Since $d^k$ is a $p$-adically continuous function in $k$ as long as $(d,p)=1$, we see that the functions $\sigma^{(p)}_{k-1}(n)$ satisfy nice $p$-adic properties.  

Specifically, let $\W = \Hom_{\cont}(\Zpx,\Cpx)$ denote weight space which is isomorphic to $p-1$ copies of the open unit disc of $\Cp$. 
Let $\W_j$ denote the disc of weights with tame character $\omega^j$ where $\omega$ is the mod $p$ cyclotomic character.  
We view $\Z \inj \W$ by identifying $k$ with the character $z \mapsto z^k$.  Then there exists rigid analytic functions $A_n$ on $\W$ such that $A_n(k) = \sigma_{k-1}^{(p)}(n)$ for each $n \geq 1$.

The constant term of $E_k^{\ord}$ also enjoys nice $p$-adic properties.  
For $k>2$ even, we have $-(1-p^{k-1}) \frac{B_k}{2k} = \frac{1}{2} \zeta_p(k)$ where $\zeta_p(\kappa)$ is the $p$-adic $\zeta$-function as in \cite{Colmez-pL,BD-evil}.  In particular, the formal $q$-expansion:
$$
\E_\kappa = \frac{1}{2} \zeta_p(\kappa) + \sum_{n \geq 1} A_n(\kappa) q^n
$$
defines the $p$-adic Eisenstein family in the weight variable $\kappa$ over all even components of weight space.

Note further that if $\kappa_0$ is a zero of $\zeta_p(\kappa)$ (which necessarily cannot be a classical weight), then the weight $\kappa_0$ Eisenstein series $\E_{\kappa_0}$ ``looks cuspidal".  In the spirit of Ribet's proof of the converse to Herbrand's theorem, one can show that there is a cuspidal Hida family which specializes at weight $\kappa_0$ to an Eisenstein series.  That is, the weights of the crossing points of the cuspidal and Eisenstein branches of the Hida family occur precisely at the zeroes of $\zeta_p(\kappa)$.

We set some notation:\ let $\T$ denote the universal ordinary Hecke algebra of tame level $N=1$ acting on ordinary modular forms of all weights (as in \cite{Hida-Iwmod}) which is a $\Lambda$-module where $\Lambda = \Zp[[1+p\Zp]]$.  Let $\Tc$ denote the quotient of $\T$ corresponding to the cuspidal Hecke algebra, and let $\mc \subseteq \Tc$ denote some maximal ideal containing the cuspidal Eisenstein ideal.  Note that the choice of such an ideal implicitly chooses some disc $\W_j$ of weight space over which our Hida family sits and for which $(p,j)$ is an irregular pair.

For simplicity and concreteness, we begin by imposing the following hypothesis:
\begin{equation}
\label{r1}
\tag{Cuspidal rank one}
\dim_{\Lambda} \Tcm = 1
\end{equation}
so that the geometry of the Hida family is as simple as possible.

We note that this condition is equivalent to there being a unique eigenform $f_k$ in weight $k$ congruent to $E^{\ord}_{k}$ for one (equivalently any) $k \equiv j \pmod{p-1}$.
The condition \eqref{r1} certainly does not hold  for all irregular pairs $(p,k)$.  But we checked that it does hold for all irregular pairs $(p,k)$ with $p<10^5$ with the exception of $p=547$ and $k = 486$.  In this case, $\dim_{\Zp} \Tcm = 2$.

\subsection{$p$-adic $L$-functions and mod $p$ multiplicity one}
Each of the cuspforms $f_k$ have an associated $p$-adic $L$-function $L^+_p(f_k)$, which can be equivalently viewed as a $\Zp$-valued measure on $\Gammac$ or as an element in the completed group algebra $\Zp[[\Gammac]]$. The $+$ superscript here indicates that this is the $p$-adic $L$-function supported only on even characters.   Moreover, these $p$-adic $L$-functions vary $p$-adic analytically in $k$ yielding a two-variable $p$-adic $L$-function as in \cite{Kitagawa,GS}.   

Under \eqref{r1}, we can view the two-variable $p$-adic $L$-function $L_{p}^+(\kappa)$ as an element  in $\Zp[[\Gammawj \times \Gammac]]$ where $\Gammawj = 1+p\Zp$. Here, we think of $\kappa$ as ranging over weights in $\W_j$.
That is, under this assumption, we can conflate a parameter of the Hida family with a parameter of $\W_j$.  At each classical weight in $\W_j$, the two-variable $p$-adic $L$-function specializes to a one-variable $p$-adic $L$-function.  That is, for $k \in  \Z^{\geq 2} \cap \W_j$ we have $$L^+_{p}(\kappa)\big|_{{\kappa =k}} = L^+_p(f_k).$$
One can also naturally attach $p$-adic $L$-functions to Eisenstein series (e.g.\ via overconvergent modular symbols as in \cite{GS}), and somewhat surprisingly the result is simply the zero distribution (see Theorem \ref{thm:pLvanish}).\footnote{The corresponding minus $p$-adic $L$-function on the space of odd characters does not vanish (see Remark \ref{rmk:odd}).}

The key idea of this paper is to use the vanishing of the $p$-adic $L$-functions on the Eisenstein branch and the intersection points of the Eisenstein and cuspidal branches to make deductions about the cuspidal $p$-adic $L$-functions.  Indeed, if there were a  two-variable $p$-adic $L$-function defined simultaneously over both the Eisenstein and  cuspidal branches, the vanishing of the $p$-adic $L$-function on the Eisenstein branch would imply the vanishing of the $p$-adic $L$-functions at the crossing points of the branches.
This in turn would mean that $L^+_p(f_k)$ becomes more and more divisible by $p$ as $k$ moves close to one of these crossing points --- i.e.\ to one of the zeroes of $\zeta_p(\kappa)$.  In the language of Iwasawa theory, this means that the $\mu$-invariants in the Hida family blow up as you approach these zeroes!

We now set some further notation and introduce a mod $p$ multiplicity one condition which helps explain this phenomenon.  Let $\Tk$ (resp.\ $\Tck$) denote the full Hecke algebra over $\Zp$ which acts faithfully on $M_k(\Gamma_0(p))^{\ord}$ (resp.\ $S_k(\Gamma_0(p))^{\ord}$).  Let $\m_k$  denote the maximal ideal of $\Tk$ corresponding to $E^{\ord}_k$, and let $\mck$ denote the image of $\m_k$ in $\Tck$.  Let $(p,k)$ be an irregular pair, in which case $\mck$ is a maximal ideal (i.e.\ it is a proper ideal).  We consider the following condition:
\begin{align*}
\label{m1}
\tag{Mult One}
\dim_{\Fp} \left( H^1_c(\SL_2(\Z),\cP_{k-2}^\vee \otimes \Fp)^+[\m_k] \right) = 1 
\end{align*}
where $\cP_g$ is the collection of $\Qp$-polynomials of degree less than or equal to $g$ which preserve $\Zp$.
We will see that \eqref{m1} holds for one irregular pair $(p,k)$ if and only if it holds for all irregular pairs $(p,k)$ as $k$ runs over a fixed residue class mod $p-1$ (Theorem \ref{thm:m1cong}).  
 
To see the relevance of condition \eqref{m1}, a weaker version of the argument on $\mu$-invariants given above would be to simply say that since $f_k$ and $E_k^{\ord}$ satisfy a congruence modulo $p$, one would hope for their $p$-adic $L$-functions to also satisfy a congruence.  Since the $p$-adic $L$-function of $E_k^{\ord}$ is 0, we would then deduce that the $\mu(f_k)>0$.  However, the implication that a congruence of forms leads to a congruence of $p$-adic $L$-functions is a subtle one which is typically proven via a congruence between their associated modular symbols --- that is via \eqref{m1}.  Mod $p$ multiplicity one is now well-established in the residually irreducible case (e.g.\ \cite{Wiles-FLT}), but is more subtle in the residually reducible case, and not always true for tame level greater than 1.  
But we do note that \eqref{r1} implies  \eqref{m1} --- see Lemma \ref{lemma:r1togoren} and Theorem \ref{thm:gorentom1}.  

We  now state a precise theorem relating the $\mu$-invariants of the $f_k$ with the $p$-adic $\zeta$-function.
For $a$ even, we write $\mu(f,\omega^a)$ for the $\mu$-invariant of the projection of $L^+_p(f)$ to the $\omega^a$-component of the semi-local ring $\Zp[[\Gammac]]$.

\begin{thm}
\label{thm:anlowerbndintro}
Fix an irregular pair $(p,j)$ and assume that  \eqref{r1} holds for this pair.  Then
$$
\zeta_p(\kappa) \text{~~divides~~} L^+_{p}(\kappa)
$$
in $\Zp[[\Gammawj \times \Zpx]]$.  In particular,
$$
\mu(f_k,\omega^a) \geq \ord_p(\zeta_p(k)) = \ord_p \left(  B_k/k \right)
$$
for each even $a$ with $0 \leq a \leq p-1$ and for classical $k \equiv j \pmod{p-1}$.
\end{thm}

Note that this theorem makes precise the previous claim that the $\mu$-invariants blow up as one approaches the zeroes of the $p$-adic $\zeta$-function as the above lower bound blows up as one approaches these zeroes.

Our method of proof of this theorem is to build a two-variable $p$-adic $L$-function on the full Hida family (i.e.\ on both the cuspidal and Eisenstein branches).  
We follow the construction of Greenberg and Stevens, but keep careful track over the integrality of all of the objects.  Then \eqref{m1} implies that
some large space of overconvergent modular symbols is free as a Hecke-module.  This freeness allows us to build a two-variable $p$-adic $L$-function over the full Hida algebra.

With this two-variable $p$-adic $L$-function in hand, the vanishing of the $p$-adic $L$-functions of the Eisenstein series implies the divisibility of Theorem \ref{thm:anlowerbndintro}.
We note that the last claim of the theorem follows immediately from this divisibility.  Indeed, when one specializes $\kappa$ to some classical weight $k$, one gets that the {\it number} $\zeta_p(k)$ divides $L^+_p(f_k)$ in $\Zp[[\Gammac]]$ -- which exactly gives the above lower bound on the $\mu$-invariant.

We now give an upper bound on $\mu$-invariants which holds even without our \eqref{r1} assumption.  In what follows, we write $\mu(f)$ for $\mu(f,\omega^0)$.

\begin{thm}
\label{thm:anupperbndintro}
Fix an irregular pair $(p,j)$ and assume that \eqref{m1} holds for this pair.  Then 
$$
\mu(f_k) \leq \ord_p(a_p(f_k)-1)
$$
for all classical $k \equiv j \pmod{p-1}$.
\end{thm}

 We sketch the simple proof of this theorem now.  Assuming \eqref{m1}, one shows that $L(f_k,1)/\Omega^+_f$ is a $p$-adic unit (as one knows that the plus modular symbol attached to $f$ is congruent to a boundary symbol mod $p$).  The interpolation formula for the $p$-adic $L$-function thus gives that  $L_p(f_k,{\bf 1})$ has valuation $\ord_p(a_p(k)-1)$ which in turns gives our upper bound on $\mu$. 

Thus under \eqref{r1} (which implies \eqref{m1}), we get the following string of inequalities
\begin{equation}
\label{eqn:ineq}
\ord_p(\zeta_p(k)) \leq \mu(f_k) \leq \ord_p(a_p(f_k)-1).
\end{equation}
We note that since $a_p(E_\kappa^{\ord})=1$ for all $\kappa \in \W$, we have that $\zeta_p(\kappa)$ divides $a_p(f_\kappa)-1$ in $\Zp[[\Gamma_w]]$ which is consistent with the above inequalities.

Given the philosophy that $\mu$-invariants are ``as small as they can be", it's tempting to make the following conjecture. We will also further justify this conjecture after analyzing the algebraic side of the situation.

\begin{conj}
\label{conj:anmuintro}
For an irregular pair $(p,j)$ satisfying \eqref{r1} we have
$$
\mu(f_k,\omega^a) = \ord_p(\zeta_p(k))
$$ 
for $a$ even and for any classical $k \equiv j \pmod{p-1}$.
\end{conj}

Besides giving an elegant formula for $\mu$-invariants in a Hida family, the above conjecture has an appealing philosophical feel.  Namely, the non-trivial $\mu$-invariant is explained away through $p$-adic variation.  For an isolated form, a non-trivial $\mu$-invariant almost feels like an error in the choice of the complex period.  However, in the family, one sees these $\mu$-invariants explicitly arising from the special values of some interesting analytic function.  Further, by Ferrero-Washington \cite{FW-mu}, the $\mu$-invariant of $\zeta_p(\kappa)$ is 0, and so all divisibility by a power of $p$ vanishes in the family.  We further note that Conjecture \ref{conj:anmuintro} implies the weaker statement that the two-variable $\mu$-invariant of $L_p^+(\kappa)$ vanishes.

We note that this conjecture holds for $a=0$ for all irregular pairs $(p,k)$ with $p<2000$.  Unfortunately, this provides scant evidence for the conjecture as in this range, $\ord_p(a_p(f_k)-1)=1$, and thus the lower bound equals the upper bound in the inequalities in \eqref{eqn:ineq} forcing Conjecture \ref{conj:anmuintro} to hold.

However, when $a>0$, 
 we have no a priori upper bound on the $\mu$-invariant (i.e.\ Theorem \ref{thm:anupperbndintro} does not apply) and thus we must instead compute $\mu$-invariants of $p$-adic $L$-functions to verify that our conjecture holds.  In this case, we verified our conjecture holds for $p<750$ (for all even $a$).  For details, see section \ref{sec:numerical} where we describe the extensive computation we did using the algorithms in \cite{PS-explicit} on overconvergent modular symbols.

Returning to the case of trivial tame character (i.e.\ $a=0$), when the lower and upper bounds of \eqref{eqn:ineq} meet, the $p$-adic $L$-functions of the $f_k$ turn out to be very simple.  In what follows, $L^+_p(\kappa,\omega^a)$ denotes the projection of $L^+_p(\kappa)$ to the $\omega^a$-component of $\Zp[[\Gammawj \times \Gammac]]$.

\begin{thm}
\label{thm:anfullintro}
Fix an irregular pair $(p,j)$ satisfying \eqref{r1} and for which
$$
\ord_p(\zeta_p(j)) = \ord_p(a_p(f_j)-1).
$$
Then 
$$
L^+_{p}(\kappa,\omega^0) = \zeta_p(\kappa) \cdot U
$$ 
where $U$ is a unit in $\Zp[[\Gammawj \times (1+p\Zp)]]$.
In particular, for every classical $k \equiv j \pmod{p-1}$, we have that 
$$\lambda(f_k)=0 \text{~~and~~}\mu(f_k) = \ord_p(\zeta_p(k))= \ord_p\left(  B_k/k \right).$$  
That is, up to a unit, $L^+_p(f_k,\omega^0)$ is simply a power of $p$.
\end{thm}

We sketch the proof here.  For our fixed $j$, our assumption combined with \eqref{eqn:ineq} tells us the value of the $\mu$-invariant exactly:
$$
\ord_p(\zeta_p(j)) = \mu(f_{j}) = \ord_p(a_p(j)-1).
$$
But, as argued above, we know that the value of $L^+_p(f_{j})$ at the trivial character has $p$-adic valuation equal to $\ord_p(a_p(j)-1)$.  Since this valuation equals the $\mu$-invariant, we get that $\lambda(f_{j})=0$.  Then, the divisibility of Theorem \ref{thm:anlowerbnd}, tells us that the projection of 
$$
\frac{L^+_p(\kappa)}{\zeta_p(\kappa)}\Big|_{\kappa = j}
= \frac{L^+_p(f_{j})}{\zeta_p(j)}
$$
to $\Zp[[1+p\Zp]]$ has vanishing $\mu$ and $\lambda$-invariants since $\mu(f_j) = \ord_p(\zeta_p(j))$, and is thus a unit.  Since $\frac{L^+_p(\kappa)}{\zeta_p(\kappa)}$ specializes to a unit at one weight, it must itself be a unit, as desired.

We note that it appears reasonable to believe that the case where the upper and lower bounds of \eqref{eqn:ineq} meet is the ``generic" case. Indeed, if $e = \ord_p(\zeta_p(k))$, then $p^e$ necessarily divides $a_p(f_k)-1$, and one might imagine that $(a_p(f_k)-1)/p^e$ is a random number modulo $p$.  If this is the case, then for any given $p$, the probability is $1/p$ for these bounds not to meet.  One then expects there to be infinitely many $p$ such that the bounds in \eqref{eqn:ineq} fail to meet, but given how slowly $\sum_p \frac{1}{p}$ diverges such examples will be extremely rare.

\subsection{A more general picture}

We now discuss the more general case where \eqref{r1} is no longer assumed to hold, but still for the introduction we keep the tame level $N=1$.
In trying to repeat the arguments from earlier in the introduction without assuming \eqref{r1}, several new difficulties emerge.  First, we no longer have that \eqref{m1} is automatically satisfied which is key to our construction of a two-variable $p$-adic $L$-function over both the cuspidal and Eisenstein families.  Second, the intersections of the cuspidal branches and the Eisenstein branch are no longer controlled solely by the $p$-adic $\zeta$-function.  Indeed, if there are multiple branches of the cuspidal family, some of them may not even intersect the Eisenstein family.

To remedy the first of these problems, we introduce a Gorenstein hypotheses:
\begin{equation}
\label{goren}
\tag{Goren}
\Tkm  \text{~is~Gorenstein}
\end{equation}
where $(p,k)$ is some irregular pair.  We note that \eqref{goren} holds for one irregular pair $(p,k)$ iff it holds for all $(p,k)$ as $k$ varies over a fixed residue class mod $p-1$ iff  $\T_{\m}$ is Gorenstein where $\m \subseteq \T$ is the maximal ideal attached to $E_k$.

The connection between Gorenstein hypotheses and mod $p$ multiplicity one in the residually irreducible case has been understood for a while now as in \cite{Mazur-eisenstein, Tilouine-gorenstein, Wiles-FLT}.  In our residually reducible case, we have that  \eqref{goren} implies \eqref{m1} which follows from work of Ohta \cite{Ohta-towers1,Ohta-towers2} and Sharifi \cite{Sharifi-reciprocity} (see Theorem \ref{thm:gorentom1}).  

We also note that in this more general setting, we cannot consider our cuspidal two-variable $p$-adic $L$-function as an elements in $\Zp[[\Gammawj \times \Zpx]]$ since we are not able to conflate the Hida family with weight space.  Instead, 
we will consider the two-variable $p$-adic $L$-function as an element $L^+_p(\mc)$ in $\Tcm[[\Gammac]]$.  The weight variable of $L^+_p(\m)$ then varies over the spectrum of $\Tcm$; that is, for $\p$ a height 1 prime of $\Tcm$ of residual charactertic 0, we can ``evaluate" $L^+_p(\mc)$ at $\p$ by looking at its image in $(\Tcm / \p)[[\Gammac]]$.  Here $\Tcm/\p$ is a finite extension of $\Zp$.  Moreover, if $\p_f \subseteq \Tcm$ is the prime ideal associated with a classical ordinary eigenform $f$, then the image of $L^+_p(\mc)$ in $\Tcm  / \p_f[[\Gammac]]$ recovers the single variable $p$-adic $L$-function $L^+_p(f)$.

For the second problem, to find a replacement for the $p$-adic $\zeta$-function to control the intersection of the Eisenstein and cuspidal branches, we seek a single element in the cuspidal Hecke algebra $\Tcm$ which measures congruences between Eisenstein series and cuspforms.  Essentially, by definition, we can simply use, if one exists, a generator of the cuspidal Eisenstein ideal in $\Tcm$.  That is, assume  the Eisenstein ideal of $\Tcm$ is principal generated by say $\Leis$.  Fix a weight $k$ cuspidal eigenform $f$ congruent to $E^{\ord}_k$ and let $\p_f \subseteq \Tc$ denote the corresponding prime ideal.
Then $\Tcm / \p_f$ is some finite extension of $\Zp$.  Write $\O_f$ for the ring of integers in the field of fractions of $\Tcm / \p_f$, and write $\Leis(\p_f)$ to 
denote the image of $\Leis$ in $\O_f$.  We then have that $E^{\ord}_k$ is congruent to $f$ modulo $\Leis(\p_f) \O_f$.

Unfortunately, there is no a priori reason why the Eisenstein ideal is principal.  To ensure the principality of this ideal we introduce another Gorenstein condition:
\begin{equation}
\label{cuspgoren}
\tag{Cusp Goren}
\Tckm  \text{~is~Gorenstein}
\end{equation}
where $(p,k)$ is some irregular pair.  Again, by Hida theory, \eqref{cuspgoren} holds for one irregular pair $(p,k)$ iff it holds for all $(p,k)$ as $k$ varies over a fixed residue class mod $p-1$ iff  $\Tc_{\mc}$ is Gorenstein.

Now by work of Ohta \cite{Ohta-companion}, we have that  \eqref{goren} and \eqref{cuspgoren} hold if and only if the Eisenstein ideal is principal (see Theorem \ref{thm:eisenprincipal}).  We also note that Vandiver's conjecture implies that both \eqref{goren} and \eqref{cuspgoren} hold,  (see Propostion \ref{prop:vandtogoren}) 
and so we can offer no examples where these hypotheses fail when $N=1$.  However, when $N>1$ one does not expect these hypotheses to hold generally (see \cite[Corollary 1.4]{Wake-gorenstein}).

The following is our  generalization of Theorem \ref{thm:anlowerbndintro}.
Below $\varpi$ denotes a uniformizer of $\O_f$.

\begin{thm}
\label{thm:anlowerbndgen_intro}
Fix an irregular pair $(p,j)$ and assume that  \eqref{goren} and \eqref{cuspgoren} hold for this pair, and let $\mc$ denote the corresponding maximal ideal of $\Tc$.  Let $\Leis$ denote a generator of the cuspidal Eisenstein ideal of $\Tcm$.  Then
$$
\Leis  \text{~~divides~~} L^+_p(\mc)
$$
in $\Tc_{\mc}[[\Gammac]]$.  In particular,
$$
\mu(f,\omega^a) \geq \ord_\varpi(\Leis(\p_f))
$$
for all even $a$ with $0 \leq a \leq p-2$,  and all $f$ in the Hida family corresponding to $\mc$.  
\end{thm}

We again conjecture that the above inequality is in fact an equality (see Conjecture  \ref{conj:anmu}).  Further, this lower bound meets the upper bound of Theorem \ref{thm:anupperbndintro} if and only if the cuspidal Eisenstein ideal is generated by $U_p-1$.  In this case, our Gorenstein hypotheses are automatically satisfied 
and we have the following generalization of Theorem \ref{thm:anfullintro} which asserts that the Iwasawa theory in this case is as simple as it can be.

\begin{thm}
\label{thm:anfullintrogen}
For an irregular pair $(p,j)$ for which $U_p-1$ generates the cuspidal Eisenstein ideal of $\Tcm$, we have 
$$\lambda(f)=0 \text{~~and~~}\mu(f) = \ord_p(a_p(f)-1)$$  
for all $f$ in the corresponding Hida family.
\end{thm}

\subsection{Selmer groups}

We return the situation where $(p,j)$ is an irregular pair which satisfies \eqref{goren} and \eqref{cuspgoren} and let $\Leis$ denote a generator of the associated cuspidal Eisenstein ideal.  Let $f$ denote some classical form which belongs to the corresponding Hida family.

On the algebraic side, we note that there is not a well-defined Selmer group attached to  $f$.  Indeed, within the Galois representation $\rho_{f} : G_{\Q} \to \Aut(V_{f})$,  one must choose a Galois stable lattice $T \subseteq V_{f}$.  One then attaches a Selmer group to $V_{f}/T$;  we denote this Selmer group by $\Sel(\Qinf,V_f/T) \subseteq H^1({\Qinf},V_{f}/T)$ where $\Qinf$ is the cyclotomic $\Zp$-extension.  Changing the lattice $T$ can change the $\mu$-invariant of the corresponding Selmer group.

In this setting, we have the following collection of lower bounds on the possible algebraic $\mu$-invariants that can occur.

\begin{thm}
\label{thm:alglowerbnd}
Let $(p,j)$ denote an irregular pair satisfying \eqref{goren} and \eqref{cuspgoren} and let $f$ denote a classical eigenform in the corresponding Hida family.  Then there exists a chain of lattice $T_0 \subseteq T_1 \subseteq \cdots \subseteq T_m$ in $V_f$ with $m=\ord_\varpi(\Leis(\p_f))$ such that
$$
\mu(\Sel(\Qinf,V_{f}/T_r)^\vee) \geq r
$$
for $0 \leq r \leq m$.
\end{thm}

This theorem is proven as follows:\ for each $r$ in the above range, there is a lattice $T_r$ such that $T_r/p^r T_r$ is a reducible representation with a cyclic submodule of size $p^r$ which is odd and ramified at $p$.  Using Greenberg's methods as in \cite[Proposition 5.7]{Greenberg-CIME}, the lower bound on $\mu$ follows.

Greenberg, in the case of elliptic curves, would then conjecture that this lower bound on $\mu$ is actually an equality (see \cite[Conjecture 1.11 and page 70]{Greenberg-CIME}).  Combining this conjecture with Theorem \ref{thm:anlowerbndintro}, we see that if there is any main conjecture defined over $\Zp$ for $f$, we must be using the lattice $T_r$ with maximal $\mu$-invariant.  Set $T$ equal to such a lattice and let $A_{f} = V_{f}/T$.

We now state an upper bound for the algebraic $\mu$-invariant.  To do so, we need to invoke Vandiver's conjecture.  Namely, for $r$ even, we denote the $\omega^r$-part of Vandiver's conjecture by:
\begin{equation}
\label{vand}
\tag{$\text{Vand~}\omega^r $}
\text{the~}\omega^{r}\text{-eigenspace~of~}\Cl(\Q(\mu_p))[p]{~vanishes}.
\end{equation}

\begin{thm}
\label{thm:algupperbnd}
Let $(p,j)$ be an irregular pair for which $\vandv{2-j}$ holds.
For $f$ a classical form in the corresponding Hida family, we have
$$
\mu(\Sel(\Qinf,A_{f})^\vee) \leq \ord_p(a_p(f)-1).
$$
\end{thm}

 Our method of proof is analogous to the proof of Theorem \ref{thm:anupperbndintro}.  There we bounded the $\mu$-invariant of the $p$-adic $L$-function by computing the $p$-adic valuation of its special value at the trivial character.  Here we instead look at $\Sel(\Qinf,A_f)^\Gamma$ and by bounding its size we get a bound on the $\mu$-invariant (see  Lemma \ref{lemma:mu}).  To bound the size of $\Sel(\Qinf,A_f)^\Gamma$, we use a control theorem 
 to compare this group with $\Sel(\Q,A_f)$ which we in turn control via Vandiver's conjecture.

We now state the algebraic analogue of Conjecture \ref{conj:anmuintro} for our chosen lattice, and note that this is essentially Greenberg's conjecture on $\mu$-invariants in this special case.

\begin{conj}
\label{conj:algmu}
For an irregular pair $(p,j)$ satisfying \eqref{goren} and \eqref{cuspgoren} and $f$ a classical form in the corresponding Hida family, we have
$$
\mu(\Sel(\Qinf,A_{f})^\vee) = \ord_p(\Leis(\p_f)).
$$ 
\end{conj}

Lastly, we state a theorem in the case where the lower and upper bounds meet.

\begin{thm}
\label{thm:algfull}
For an irregular pair $(p,j)$ such that $U_p-1$ generates the corresponding cuspidal Eisenstein ideal, we have
$$
\Sel(\Qinf,A_{f})^\vee \cong  \Lambda/(a_p(f)-1) \Lambda
$$
for each classical $f$ in the corresponding Hida family.
In particular,
$$
\mu(\Sel(\Qinf,A_{f})^\vee) = \ord_p(a_p(f)-1) \text{~~~and~~~}
\lambda(\Sel(\Qinf,A_{f})^\vee) = 0.
$$
\end{thm}

We note that we do not need to assume Vandiver's conjecture for the above theorem as our assumption that $U_p-1$ generates the cuspidal Eisenstein ideal implies that the relevant class group vanishes by the results in \cite{Wake-gorenstein}.

Combining Theorems \ref{thm:anfullintro} and \ref{thm:algfull} yields the following result.

\begin{cor}[The main conjecture]
For an irregular pair $(p,j)$ satisfying $\vandv{2-j}$ and such that $U_p-1$ generates the corresponding cuspidal Eisenstein ideal, we have
$$
\chr_{\Lambda}(\Sel(\Qinf,A_{f})^\vee) = L^+_p(f,\omega^0) \Lambda
$$
for each classical $f$ in the corresponding Hida family.
\end{cor}

We note that the above main conjecture over $\Lambda[1/p]$ follows from Kato's work \cite[Theorem 17.4.2]{Kato} as the $p$-adic $L$-function is a unit up to a power of $p$.  However, neither \cite{Kato} nor \cite{Skinner-Urban} control what happens with the algebraic and analytic $\mu$-invariants in the residually reducible case.

The organization of the paper is as follows:\ in the following section, we recall Stevens' theory of overconvergent modular symbols and its connection to $p$-adic $L$-functions.  The heavier lifting of constructing two-variable $p$-adic $L$-functions is relegated to the appendix.  In the third section, we carry out our analysis of analytic $\mu$-invariants sketched in the introduction.  In the fourth section, we recall the basic definitions and facts about Selmer groups.  In the fifth section, we carry out our analysis of algebraic $\mu$-invariants.

~\\

\noindent
{\bf Acknowledgements:} We heartily thank Preston Wake for his help and advice throughout this project especially regarding his extensive help with the arguments involving the Gorenstein properties of Hecke algebras.

\section{Overconvergent modular symbols and $p$-adic $L$-functions}

In this section, we recall Glenn Stevens' theory of overconvergent modular symbols and its connection to $p$-adic $L$-functions (see \cite{Stevens-rigid} as well as \cite{PS-explicit}).  

\subsection{Modular symbols}

Let $\Delta_0$ denote the space of degree zero divisors on $\P^1(\Q)$, and let $V$ be some right $\Z[\Gamma]$-module where $\Gamma = \Gamma_1(N)$.  We define the space of $V$-valued modular symbols of level $\Gamma$ to be
the collection of additive maps
$$
\label{eqn:ms}
\Hom_{\Gamma}(\Delta_0,V) := \{ \varphi : \Delta_0 \to V ~|~ \varphi(\gamma D) = \varphi(D) | \gamma \text{~for~all~}\gamma \in \Gamma\text{~and~}D \in \Delta_0\}.
$$
Here $\Gamma$ acts on $\Delta_0$ on the left via linear fractional transformations $\smallmat \cdot z = \frac{az+b}{cz+d}$.  

We will also be interested in a subspace of {\it boundary} modular symbols:\ $\Hom_{\Gamma}(\Delta,V)$ where $\Delta = \Div(\P^1(\Q))$.  Note that these boundary symbols naturally map to $V$-valued modular symbols of level $\Gamma$ by restriction to $\Delta_0$.  

If $V$ is endowed with an action of a larger collection of matrices, one can define a  natural Hecke action on these spaces of modular symbols.  For instance, for a prime $p$, consider the semi-group
$$
S_0(p) := \{  \gamma = \smallmat \in M_2(\Z) \text{~such~that~} p \nmid a, p |c, \text{~and~}\det(\gamma) \neq 0
\}.
$$
If $V$ is a $\Z[S_0(p)]$-module, then $\Hom_{\Gamma}(\Delta_0,V)$ is naturally a Hecke-module.

If the order of every torsion element of $\Gamma$ acts invertibly on $V$, then \cite[Proposition 4.2]{AS} yields a canonical isomorphism
$$
\Hom_{\Gamma}(\Delta_0,V) \cong H^1_c(\Gamma,V).
$$
Further, this map is Hecke-equivariant when Hecke operators are defined on these spaces.
In what follows, we will often tacitly identify spaces of modular symbols with compactly supported cohomology groups.

For $F$ a field and $g \geq 0$, let $\cP_g(F) \subseteq F[z]$ denote the subset of polynomials with degree less than or equal to $g$.  Set $\cP^\vee_g(F) = \Hom(\cP_g(F),F)$.  We endow $\cP_g(F)$ with a left action of $S_0(p)$ by 
$$
(\gamma \cdot P)(z) = (a+cz)^g P \left(\frac{b+dz}{a+cz} \right)
$$
for $P \in \cP_g(F)$ and $\gamma \in \GL_2(\Q)$, and equip $\cP^\vee_g(F)$ with a right action by
$$
(\alpha|\gamma)(P) = \alpha(\gamma \cdot P)
$$
for $\alpha \in \cP^\vee_g(F)$.

One can associate to each eigenform $f$ in $S_k(\Gamma,\C)$ a $\cP_{k-2}^\vee(\C)$-valued modular symbol $\xi_f$ of level $\Gamma$ defined by
$$
\xi_f(\{r\} - \{s\})(P(z)) = 2 \pi i \int_s^r f(z) P(z) dz
$$
where $r,s \in \P^1(\Q)$; here we write $\{r\}$ for the divisor associated to $r \in \Q$.
The symbol $\xi_f$ is a Hecke-eigensymbol with the same Hecke-eigenvalues as $f$.

Since the matrix $\iota := \psmallmat{-1}{0}{0}{1}$ normalizes $\Gamma$, it acts as an involution on these spaces of modular symbols. Thus $\xi_f$ can be uniquely written as $\xi^+_f + \xi^-_f$ with $\xi^\pm_f$ in the $\pm1$-eigenspace of $\iota$.  By a theorem of Shimura \cite{Shimura}, there exists complex numbers $\Omega^\pm_f$ and a number field $K$ such that for each $D \in \Delta_0$, $\xi^\pm_f(D)$ takes values in $K\Omega^\pm_f$.    We can thus view
$\varphi^\pm_f := \xi^\pm_f / \Omega^\pm_f$ as taking values in 
$\cP^\vee_{k-2}(K)$, and for a fixed embedding $\Qbar \inj \Qpbar$, we can view $\varphi^\pm_f$ as taking values in $\cP^\vee_{k-2}(\Qpbar)$.

As we are interested in $\mu$-invariants in this paper, we must carefully normalize our choice of periods.  To this end, we will choose $\Omega_f^\pm$ so that for all $D \in \Delta_0$, all values of $\varphi_f^\pm(D) \in \cP_{k-2}^\vee(\Qpbar)$ are $p$-adic integers.  Further, we insist that there is at least one divisor $D$ so that $\varphi_f^\pm(D)$ takes on at least one value which is a $p$-adic unit.  
Periods $\Omega_f^\pm$ which achieve this normalization we will call canonical periods.

\begin{remark}
We note that this definition differs slightly from the one given in \cite{Vatsal-canonical}.  where parabolic cohomology is used rather than compactly supported cohomology.
\end{remark}

\subsection{Measures and $p$-adic $L$-functions}

Let $G$ denote a $p$-adic Lie group which for this paper we will be taking to be either $\Zp$, $\Zpx$ or $\X$.
Let $\Cont(G)$ denote the space of continuous maps from $G$ to $\Zp$, and let $\Meas(G)$ denote the continuous $\Zp$-dual of $\Cont(G)$ which we regard as the space of $\Zp$-valued measures on $G$.  

For $\mu \in \Meas(G)$ and $U$ a compact open of $G$, we write $\mu(U)$ for $\mu({\bf 1}_U)$ where ${\bf 1}_U$ is the characteristic function of $U$.  Since the $\Zp$-span of these characteristic functions are dense in $\Cont(G)$, a measure is uniquely determined by its values on the compact opens of $G$.
We further note that there is a natural isomorphism $\Meas(G) \cong \Zp[[G]]$
which sends the Dirac-measure $\delta_g$ supported at $g \in G$ to the group-like element $[g]$ in $\Zp[[G]]$.  

Set $\Gamma_0 = \Gamma_0(p) \cap \Gamma_1(N)$, and let $f$ be a $p$-ordinary eigenform in $S_k(\Gamma_0,\Qpbar)$;  that is, if $a_p(f)$ is the $p$-th Fourier coefficient of $f$, then $a_p(f)$ is a $p$-adic unit.  
Then $L_p(f)$, the $p$-adic $L$-function of $f$,  is an element of $\Mx \otimes \O$ where $\O := \O_f$ is the subring of $\Qpbar$ generated over $\Zp$ by the Hecke-eigenvalues of $f$.  Explicitly, we define $L_p(f)$ via the following formula:
$$
L^\pm_p(f)(a+p^n\Zp)
 := \frac{1}{a_p(f)^n} \varphi^\pm_f(\{\infty\} - \{a/p^n\})(1),
$$
and $L_p(f) = L_p^+(f) + L_p^-(f)$.
The fact that this formula for $L_p(f)$ defines a measure follows from the fact that $\varphi^\pm_f$ is a $U_p$-eigensymbol with eigenvalue $a_p(f)$.\footnote{We note that $L_p(f)$ implicitly depends upon the choice of the periods $\Omega_f^\pm$ and so it is an abuse of language to call it {\it the} $p$-adic $L$-function of $f$.}

If $\chi$ denotes a Dirichlet character of conductor $p^n$, then the $p$-adic $L$-function satisfies the interpolation property:
\begin{equation}
\label{eqn:interp}
\int_{\Zpx} \chi ~dL_p(f) = \begin{cases} \displaystyle
\frac{1}{a_p(f)^n} \cdot  \frac{p^n}{\tau(\chi^{-1})} \cdot \frac{L(f,\chi^{-1},1)}{\Omega^\ve_f} & n>0 \\
~\\
\displaystyle\left( 1- \frac{1}{a_p(f)}\right) \frac{L(f,1)}{\Omega^+_f} & n=0
\end{cases}
\end{equation}
where $\ve$ equals the sign of $\chi(-1)$.  

\subsection{Overconvergent modular symbols and $p$-adic $L$-functions}
\label{sec:OMS}

We endow the space $\M$ with a weight $g$ action of $S_0(p)$ via the formula
$$
(\mu |_g \gamma)(f) = \mu\left(  (a+cz)^g f\left( \frac{b+dz}{a+cz} \right)\right)
$$
where $\gamma = \smallmat$ and $f$ is a continuous function on $\Zp$.  When equipped with this action, we denote this space of measures by $\Mv{g}$. 

Let $\cP_g := \cP_g(\Zp) \subset \cP_g(\Qp)$ denote the continuous $\Zp$-valued functions on $\Zp$ which are given by $\Qp$-polynomials of degree less than or equal to $k$.  
This space is generated over $\Zp$ by the binomial coefficients $\binom{z}{j}$ for $0 \leq j \leq k$ (see Theorem \ref{thm:Mahler}).
Set $\cP_g^\vee = \Hom(\cP_g,\Zp)$.

As $\cP_g$ naturally sits inside of $\Cont(\Zp)$, restriction yields a $S_0(p)$-equivariant map 
$$
\Mv{g} \to \cP_g^\vee,
$$
and thus a map
$$
H^1_c(\Gamma_0,\Mv{g}) \lra H^1_c(\Gamma_0,\cP_g^\vee).
$$
We refer to both of these maps as specialization.  

If $X$ is a Hecke-module with an action of $U_p$, we define $X^{\ord}$ to be the intersection of the image of all powers of $U_p$.  The following is Stevens' control theorem in the ordinary case (see \cite{Stevens-rigid,PS-explicit,PS-critical}).

\begin{thm}[Stevens]
\label{thm:control}
For $g \geq 0$, specialization induces the isomorphism
$$
H^1_c(\Gamma_0,\Mv{g})^{\ord} \otimes \Qp \stackrel{\sim}{\lra} H^1_c(\Gamma_0, \cP_g^\vee)^{\ord} \otimes \Qp.
$$
Moreover, if $\Phi^\pm \in H^1_c(\Gamma_0,\Mv{k-2}) \otimes \Qpbar$ is the unique lift of $\varphi^\pm_f$, then $$\Phi^\pm(\{ \infty \} - \{ 0 \}){\big|_{\Zpx}}$$
is the $p$-adic $L$-function of $f$ -- i.e.\ it satisfies the interpolation property in (\ref{eqn:interp}) for some choice of canonical periods $\Omega_f^\pm$.
\end{thm}

\subsection{$p$-adic $L$-functions of ordinary Eisenstein series}
\label{sec:boundary}

Fix a primitive character $\psi : (\Z/N\Z)^\times \to \C^\times$ and consider the Eisenstein series
$$
E^{\ord}_{k,\psi} = 
 -(1-\psi(p)p^{k-1}) \frac{B_{k,\psi}}{2k} + \sum_{n=1}^\infty \left( \sum_{d|n, \\ p \nmid d} \psi(d) d^{k-1} \right) q^n.
$$

The following theorem describes  overconvergent modular symbols with the same system of Hecke-eigenvalues as $E^{\ord}_{k,\psi}$.

\begin{thm}
\label{thm:pLvanish}
Let $\Phi$ be any element of $H^1_c(\Gamma_0,\Mv{k-2})$ satisfying
$$
\Phi | U_q = \Phi \text{~~~for~} q \mid Np
$$
and
$$
\Phi | T_\ell = (1 + \psi(\ell) \ell^{k-1}) \Phi \text{~~~for~} \ell \nmid Np
$$
for $\psi$ be a Dirichlet character of conductor $N$.  
Then
\begin{enumerate}
\item $\Phi$ is a boundary symbol;
\item $\Phi$ is in the plus-subspace (i.e.\ $\Phi | \iota = \Phi$);
\item $\Phi(\{\infty\} - \{0\})$ is a constant multiple of $\delta_0$, the Dirac distribution at 0.
\end{enumerate}
\end{thm}

\begin{proof}
Our proof relies heavily on \cite{BD-evil}.   Let $\phi$ denote the specialization of $\Phi$ to $H^1_c(\Gamma_0, \cP_{k-2}^\vee)^{\ord}$.  Since $\phi$ is an Eisenstein symbol, standard descriptions of classical spaces of modular symbols (as in \cite[Proposition 2.5]{BD-evil}) show that $\phi$ is a boundary symbol while \cite[Proposition 2.9]{BD-evil} shows that $\phi$ is in the plus-subspace. Then \cite[Proposition 5.7]{BD-evil} shows that $\Phi$ itself must have been a boundary symbol in the plus subspace.  Lastly \cite[Lemma 5.1]{BD-evil} shows that $\Phi(\{\infty\})=0$ while \cite[Proposition 5.2]{BD-evil} shows that $\Phi(\{0\})$ is a constant multiple of $\delta_0$.
\end{proof}

\begin{remark}
\label{rmk:odd}
\begin{enumerate}
\item
One can explicitly write down such Eisenstein symbols.  This is done in great detail in \cite[section 5.2]{BD-evil}.
\item
In light of Theorem \ref{thm:control} and Theorem \ref{thm:pLvanish}, the natural $p$-adic $L$-function to attach to $E_{k,\psi}^{(p)}$ on even components of weight space is simply 0 as the restriction of $\delta_0$ to $\Zpx$ vanishes.
\item 
On odd components of weight space, the above approach does not suggest what $p$-adic $L$-function to attach to this Eisenstein series as there is no overconvergent modular symbol with these Eisenstein eigenvalues in the minus subspace.  However, in \cite{BD-evil}, a {\it partial} overconvergent modular symbol with the correct eigenvalues was constructed in the minus subspace and its associated $p$-adic $L$-function is (naturally enough) a product of $p$-adic $L$-functions of characters.
\end{enumerate}
\end{remark}

\section{Analytic results}
\label{sec:analytic}

\subsection{Notations and the like}
\label{sec:notation}
Fix a prime $p \geq 5$ and a tame level $N$.  We assume for the remainder of the paper that $p \nmid \varphi(N)$.
Let $\T$ denote the universal ordinary Hecke algebra of tame level $\Gamma_1(N)$, and let $\Tc$ denote its cuspidal quotient.    Both $\T$ and $\Tc$ are modules over $\Zp[[\Zpx \times (\Z/N\Z)^\times]]$ via the diamond operators.  For $a \in \Zpx$, we write $\langle a \rangle$ to denote the corresponding group-like element of $(a,a)$ in $\Zp[[\Zpx \times (\Z/N\Z)^\times]]$.  Further, let $\Tk$ (resp.\ $\Tck$) denote the full Hecke algebra over $\Zp$ which acts faithfully on $M_k(\Gamma_1(N) \cap \Gamma_0(p))^{\ord}$ (resp.\ $S_k(\Gamma_1(N) \cap \Gamma_0(p))^{\ord}$).  

Set $\Ieis$ equal to the Eisenstein ideal of $\T$; that is, the ideal generated by $T_\ell -(1+ \langle \ell \rangle \ell^{-1})$ for $\ell \nmid Np$ and by $U_q-1$ for $q \mid Np$.  Let $\Iceis$ denote the image of $\Ieis$ in $\Tc$ which we will call the cuspidal Eisenstein ideal.  We also denote by $\Ikeis$ (resp.\ $\Ickeis$) for the image of $\Ieis$ in $\Tk$ (resp.\ $\Tck$).

Let $\mc \subseteq \Tc$ be a maximal ideal and write $\m$ for its pre-image in $\T$.  
We note that a choice of a maximal ideal $\m$ of $\T$ distinguishes an even character of $(\Z/Np\Z)^\times$ in the following way.  The components of the semi-local ring $\Zp[[\Zpx \times (\Z/N\Z)^\times]]$ are indexed by the characters of $(\Z/pN\Z)^\times$ since $p \nmid \varphi(N)$, and the restriction of $\m$ to this ring is a maximal ideal which cuts out one of these components.  Write $\theta := \theta_{\m}$ for the corresponding character of $(\Z/pN\Z)^\times$ and write 
$\theta_{\m} = \omega^{j(\m)} \psi_{\m}$ with $0 \leq j(\m) < p-1$ and $\psi := \psi_{\m}$ a character on $(\Z/N\Z)^\times$.  Note that $(-1)^{j(\m)} = \psi(-1)$ and that $\theta$ is an even character.

Write $\m_k$ and $\mck$ for the image of $\m$ in $\T_k$ and $\Tck$ respectively.  These images are maximal as long as $k \equiv j(\m) \pmod{p-1}$.  We say that these ideals, $\m$, $\mc$, $\m_k$, $\mck$, are {\it Eisenstein} if $\m \supseteq \Ieis$ (equivalently $\mc \supseteq \Iceis$).
We note that the existence of a maximal ideal $\m \supseteq \Ieis$ is equivalent to $\ord_p( B_{j(\m),\psi_{\m}}) > 0$.

Consider the following conditions:
\begin{itemize}
\item $\psi_{\m}$ is a {\it primitive} character of conductor $N$;
\item $j(\m) = 1 \implies \psi_{\m}(p) \neq 1$;
\end{itemize}
We say $\m$ satisfies \fk~ if it is Eisenstein and the above two conditions hold.
The main reason for these conditions (along with $p \nmid \varphi(N)$) is that they ensure that there is a unique Eisenstein component in the families we are considering.  This is verified in the following lemma.

\begin{lemma}
\label{lemma:unique_eisen}
If $\m \subseteq \T$ is an maximal ideal satisfying \fk, then the subspace of Eisenstein series in $M_k(\Gamma_0,\psi_{\m})_{\m}^{\ord}$ is one-dimensional for $k \geq 2$ and $k \equiv j(\m) \pmod{p-1}$.
\end{lemma}

\begin{proof}
Clearly $E^{\ord}_{k,\psi_\m}$ is in this space.  Assume there is another Eisenstein series $E$ in this space which is an eigenform.  Then there exists characters $\chi_1$ and $\chi_2$ with conductors $f_1$ and $f_2$ such that $f_1 f_2 | Np$ and $E$ has eigenvalues $\chi_1(\ell)+\chi_2(\ell) \ell^{k-1}$ for $\ell \nmid \frac{Np}{f_1 f_2}$.
For each such $\ell$, we then have congruences
$$
\chi_1(\ell)+\chi_2(\ell) \omega^{k-1}(\ell) \equiv 1 + \psi_\m(\ell) \omega^{k-1}(\ell) \pmod{\p}
$$
where $\p$ is the maximal ideal of $\Zp[\mu_{\varphi(N)}]$.   Thus by the linear independence of characters and the fact that $p \nmid \varphi(N)$ yields that  either:\ (a) $\chi_1 = 1$ and $\chi_2 = \psi_\m$, or (b) $\chi_1 = \psi_\m \omega^{k-1}$ and $\chi_2 = \omega^{1-k}$.

In case (a), we see that $f_1=1$ while $f_2=N$ since $\psi_\m$ is primitive.  Thus $E$ and $E^{\ord}_{k,\psi_\m}$ have the same eigenvalues at all primes $\ell$ except possibly for $\ell=p$.  But since $E$ is ordinary, this forces its $U_p$-eigenvalue to be $\chi_1(p) = 1$ and $E$ and $E^{\ord}_{k,\psi_\m}$ have the same eigenvalues at all primes.

In case (b), since $f_1 f_2 | Np$, we must have that $k \equiv 1 \pmod{p-1}$, in which case $j(\m)=1$, $f_1=N$ and $f_2=1$.  Again since $E$ is ordinary we must have that its $U_p$-eigenvalue is $\psi_\m(p)$ and that $\psi_\m(p) \equiv 1 \pmod{\p}$.  
But since $p \nmid \varphi(N)$, this implies that $\psi_\m(p)=1$ which contradicts $\m$ satisfying \fk.
\end{proof}

To match with the notation of the introduction, we note that if $N=1$ and $(p,j)$ is an irregular pair, then there is a unique maximal ideal $\m$ containing $\Ieis$ with $j(\m) \equiv j \pmod{p-1}$.   Moreover \fk~ is vacuous in this case.

\subsection{Gorenstein conditions and the principality of the Eisenstein ideal}
\label{sec:goren}

We continue to use the notation of section \ref{sec:notation}. We thank Preston Wake for his help with many of the proofs in this section.

\begin{prop}
\label{prop:gorencong}
Fix a maximal ideal $\m \subseteq \T$.  The following are equivalent:
\begin{enumerate}
\item $\Tkm$ (resp.\ $\Tckm$) is Gorenstein for some $k \equiv j(\m) \pmod{p-1}$, $k \geq 2$;
\item $\Tkm$ (resp.\ $\Tckm$) is Gorenstein for all $k \equiv j(\m) \pmod{p-1}$, $k \geq 2$;
\item $\Tm$ (resp.\ $\Tcm$) is Gorenstein.
\end{enumerate}
\end{prop}

\begin{proof}
Let $\p_k$ denote the principal ideal of $\Zp[[\Zpx]]$ generated by $[\gamma]-\gamma^k$ where $\gamma$ is a topological generator of $\Zpx$.  By Hida's control theorem \cite[Corollary 3.2]{Hida-Iwmod}, we have $\Tm /\p_k\Tm \cong \Tkm$.  Thus $\Tm$ is Gorenstein if and only if $\Tkm$ is Gorenstein.  The same argument also applies to the cuspidal Hecke algebras.
\end{proof}

\begin{defn}
Let $\m \subseteq \T$ denote a maximal ideal.  We say that $\m$ (resp.\ $\mc$) satisfies \eqref{goren} (resp.\ \eqref{cuspgoren}) if any of the equivalent conditions of Proposition \ref{prop:gorencong} hold for $\Tm$ (resp.\ $\Tcm$).
\end{defn}

\begin{lemma}
\label{lemma:equiv}
If $\m \subseteq \T$ satisfies \fk, we have $\Ieis_{\m} \cong \Iceis_{\mc}$ as $\Tm$-modules and $\Ieis_{k,\m} \cong \Iceis_{k,\mc}$ as $\Tkm$-modules.
\end{lemma}

\begin{proof}
Let $K$ be the kernel of the natural surjection $\Ieis_{k,\m} \surj \Iceis_{k,\mc}$.  By  definition $K$ annihilates $S_k(\Gamma_0,\psi_{\m})_{\mk}$.  Further, by Lemma \ref{lemma:unique_eisen}, $\Ieis_{k,\m}$ annihilates all Eisenstein series in $M_k(\Gamma_0,\psi_{\m})_{\mk}$.  Thus $K$ annihilates all cuspforms and Eisenstein series, and hence $K=0$.  Then $\Ieis_{k,\m} \cong \Iceis_{k,\mc}$ for all classical $k$ immediately implies that the natural map $\Ieis_{\m} \surj \Iceis_{\m}$ is an isomorphism as well.
\end{proof}

\begin{thm}
\label{thm:eisenprincipal}
Let $\m$ denote an Eisenstein maximal ideal of $\T$.  The following are equivalent:
\begin{enumerate}
\item \label{eis:1} $\m \subseteq \T$ satisfies \eqref{goren} and \eqref{cuspgoren},
\item \label{eis:2} $\T_\m$ and $\Tc_{\mc}$ are complete intersections,
\item \label{eis:3} $\Ieis_{\m}$ is principal,
\item \label{eis:4} $\Iceis_{\mc}$ is principal,
\item \label{eis:5} $\Ieis_{k,\m_k}$ is principal for some $k \equiv j(\m) \pmod{p-1}$,
\item \label{eis:6} $\Iceis_{k,\mck}$ is principal for some $k \equiv j(\m) \pmod{p-1}$.
\end{enumerate}
\end{thm}

\begin{proof}
By Lemma \ref{lemma:equiv}, we have (\ref{eis:3}) $\iff$ (\ref{eis:4}) and (\ref{eis:5}) $\iff$ (\ref{eis:6}).
Also, (\ref{eis:3}) $\implies$ (\ref{eis:5}) is clear and (\ref{eis:2}) $\implies$ (\ref{eis:1}) is clear.  Further, it is a theorem of Ohta \cite[Theorem 2]{Ohta-companion} that (\ref{eis:1}) $\implies$ (\ref{eis:3}). 

Thus, it suffices to check that  (\ref{eis:6}) $\implies$ (\ref{eis:2}).  To this end, first note that $\Tm$ is a complete intersection if and only if $\Tkm$ is a complete intersection if and only if $\Tkm/p\Tkm$ is a complete intersection.  Further, since $\m_k = \Ieis_{k,\m_k} + p \Tkm$ and $\Ieis_{k,\m_k}$ is principal by Lemma \ref{lemma:equiv}, we have that the maximal ideal of $\Tkm/p\Tkm$ is principal.  But then $\Tkm/p\Tkm$ is a finite-dimensional local $\Fp$-algebra with principal maximal ideal.  We must have then that $\Tkm/p\Tkm \cong \Fp[x] / (x^r)$ for some $r \geq 0$, and is thus a complete intersection.   An identical argument works for $\Tckm$ as well.
\end{proof}

The conditions \eqref{goren} and \eqref{cuspgoren} do not hold generally.  See \cite[Corollary 1.4]{Wake-gorenstein} for an example where \eqref{goren} fails. However, there are no known counter-examples to \eqref{goren} or \eqref{cuspgoren} when the tame level $N=1$.  Indeed, we will verify that these two conditions follow from Vandiver's conjecture in this case.   
More precisely:

\begin{defn}
For $k$ even, we say that $(p,k)$ satisfies $\vandv{k}$ if
$$
\Cl(\Q(\mu_p))[p]^{(\omega^k)} = 0.
$$
\end{defn}

\begin{prop}
\label{prop:vandtogoren}
Let $N=1$ and let $(p,k)$ denote an irregular pair with corresponding maximal ideal $\m \subseteq \T$.  Then $\vandv{k}$ and $\vandv{2-k}$ imply that $\m$ satisfies \eqref{goren} and \eqref{cuspgoren}.
\end{prop}

\begin{proof}
By \cite[Theorem 0.4]{Kurihara-JNT}, $\vandv{k}$ together with $\vandv{2-k}$ imply that $\Iceis_{k,\mck} \subseteq \Tck$ is principal.  Then, by Theorem \ref{thm:eisenprincipal}, we have that $\Tm$ and $\Tcm$ are Gorenstein.
\end{proof}

\begin{remark}
We remind the reader that Vandiver's conjecture is known to hold for $p<2^{31}$ \cite{HHO}, and so \eqref{goren} and \eqref{cuspgoren} hold in these cases as well when $N=1$.
\end{remark}

\subsection{Freeness of spaces of modular symbols}

Let $\Lt = \Zp[[\Zpx]]$, $\MM = \Ms$ and consider $H^1_c(\Gamma_0,\MM)$.  
The following theorem is proven in the appendix.

\begin{thm}
\label{thm:m1cong}
Let $\m \subseteq \T$ be a maximal ideal satisfying \fk. The following conditions are equivalent:
\begin{enumerate}
\item
for some (equivalently for all) $j>2$ with $j \equiv j(\m) \pmod{p-1}$, we have
$$\dim_{\Fp} \left( H^1_c(\Gamma,\cP_{j-2}^\vee\otimes \Fp)^+[\m_j] \right) = 1,$$
\item
$\Hom_{\Lt}(H^1_c(\Gamma_0,\MM)^+_{\m},\Lt)$ is free over $\Tm$.
\end{enumerate}
\end{thm}

\begin{proof}
First note that \fk~implies the condition (\ep) from the appendix.  Indeed, \fk~implies by Lemma \ref{lemma:unique_eisen} that $E_{k,\psi_{\m}}$ spans the Eisenstein subspace of $M_k(\Gamma_0,\psi_{\m})_{\mk}$.  Since the boundary symbol associated to $E_{k,\psi_{\m}}$ is in the plus-subspace (see \cite[Proposition 2.9]{BD-evil}),  (\ep) follows.  Thus, this theorem follows immediately from Theorem \ref{thm:m1impliesfreeplus} after we note that it is fine to replace $\Gamma_0$ with $\Gamma$ since there are no ordinary $p$-new forms in weights greater than 2.
\end{proof}

If $\m \subseteq \T$ satisfies either of the above conditions, we say $\m$ satisfies \eqref{m1}.  Now let
$$
\tilde{H}^1(N) := \varprojlim_r H^1(Y_1(Np^r),\Zp)^{\ord}
\and
\tilde{H}^1_c(N) := \varprojlim_r H^1_c(Y_1(Np^r),\Zp)^{\ord}
$$
as in \cite[1.5-1.6]{FK}.

\begin{prop}
\label{prop:compare}
We have
$$
H^1_c(\Gamma_0,\MM)^\pm \cong \tilde{H}^1_c(N)^\pm \cong  \Hom_{\T}(\tilde{H}^1(N)^\mp,M_\Lambda)
$$
as Hecke-modules where $M_\Lambda$ is the space of $\Lambda$-adic modular forms of tame level $\Gamma_1(N)$.
\end{prop}

\begin{proof}
By \cite[Remark 3.5.10]{Ohta-congruence}, we have an identification between $\tilde{H}^1_c(N)^\pm$ and $H^1_c(\Gamma,\Meas(\D))^\pm$ where $\D = \{ (x,y) \in \Zp^2 ~|~ (x,y)=1 \}$.  Since $\Meas(\D) \cong \Ind_{\Gamma_0}^{\Gamma}(\MM)$, by Shapiro's lemma, we then have $H^1_c(\Gamma_0,\MM)  \cong \tilde{H}^1_c(N)$.
Then, the pairing in \cite[(1.6.7)]{FK} yields $
\tilde{H}^1_c(N)^\pm \cong \Hom_{\T}(\tilde{H}^1(N)^\mp,M_\Lambda)$
as desired.
\end{proof}

\begin{thm}
\label{thm:gorentom1}
Let $\m \subseteq \T$ be a maximal ideal satisfying \fk~and \eqref{goren}. Then  \eqref{m1} holds for $\m$.
\end{thm}

\begin{proof}
By \cite[Proposition 6.3.5 and (1.7.13)]{FK}, $\tilde{H}^1(N)^-_{\m}$ 
is a dualizing module for $\Tm$ and is thus free over $\Tm$ assuming \eqref{goren}.
As $(M_\Lambda)_{\m} \cong \Tm$ when $\Tm$ is Gorenstein, by Proposition \ref{prop:compare}, we have $H^1_c(\Gamma_0,\MM)^+_{\m}$ is free over $\Tm$.  Again since $\Tm$ is Gorenstein, we  have $\Hom_{\Lt}(H^1_c(\Gamma_0,\MM)^+_{\m},\Lt)$ is free over $\Tm$ as desired.
\end{proof}

\subsection{Lower bounds}

Fix a maximal ideal $\m \subseteq \T$ satisfying \fk~ and \eqref{goren}.  Then, by Theorem \ref{thm:gorentom1}, \eqref{m1} holds for $\m$, and thus the construction in Appendix \ref{sec:twovar} yields a two-variable $p$-adic $L$-function 
$L_p^+(\m)$ in $\Tm[[\Zpx]]$ defined over the full Hecke-algebra.  We now check that $L^+_p(\m) \in \Tm[[\Gammac]]$ vanishes along the Eisenstein component.

\begin{thm}
\label{thm:eisenvanish}
If  $\m \subseteq \T$ is an Eisenstein maximal ideal satisfying \fk~and \eqref{goren}, then
 $$
 L^+_p(\m) \in \Ieis_{\m}[[\Gammac]]
 $$
 where $\Ieis$ is the Eisenstein ideal of $\T$.
\end{thm}

\begin{proof}
To prove this corollary, it suffices to see that the image of $L^+_p(\m)$ vanishes in 
$\Tm/\Ieis_{\m}[[\Gammac]]$.  To see this, let $\p_{\eis,k} \subseteq \m$ denote the ideal generated by $T_\ell - (1 + \psi(\ell) \ell^{k-1})$ for $\ell \nmid Np$ and by $U_q-1$ for $q \mid Np$.  It suffices to see that the image of $L^+_p(\m)$ vanishes in $\Tm/\p_{\eis,k}[[\Gammac]]$ for all classical $k \equiv j(\m) \pmod{p-1}$.  But, by construction,
the image of $L^+_p(\m)$ in  $\Tm/\p_{\eis,k}[[\Zpx]]$ equals $\Phi(\{\infty\} - \{0\})\big|_{\Zpx}$ for some eigensymbol $\Phi$ in $H^1_c(\Gamma_0,\Mk)^+$ annihilated by $\p_{\eis,k}$.    Thus, by Theorem \ref{thm:pLvanish}, $\Phi(\{\infty\} - \{0\})\big|_{\Zpx} = 0$ as desired.
\end{proof}

Let $\m \subseteq \T$ be a maximal ideal satisfying \fk.  Set $L_p^+(\mc)$ equal to the image of $L_p^+(\m)$ in $\Tcm[[\Zpx]]$ which is the cuspidal two-variable $p$-adic $L$-function.    
If we further assume that both \eqref{goren} and \eqref{cuspgoren} hold,   by Theorem \ref{thm:eisenprincipal}, the cuspidal Eisenstein ideal $\Iceis_{\mc} \subseteq \Tc_{\mc}$ is principal generated by say $\Leispsi$.
The following theorem gives a lower bound on the $\mu$-invariants of the forms in the Hida family parametrized by $\mc$ in terms of $\Leispsi$.  As we will be imposing the following three hypotheses in much of what follows, let us say that $\m \subseteq \T$ satisfies \all~ if $\m$ satisfies \fk, \eqref{goren} and \eqref{cuspgoren}.

Before stating our theorem on $\mu$-invariants, we first define how these invariants are normalized.  

\begin{defn}
\label{defn:mu}
Let $\O$ be a finite extension of $\Zp$ and let $\Lambda_{\O} = \O[[1+p\Zp]]$.  Let $\varpi$ be a uniformizer of the integral closure of $\O$.  For a non-zero element $h \in \Lambda_{\O}$, we define $\mu(f)$ to be $n$ where $n$ is the largest integer such that $h \in \varpi^n \Lambda_{\O} - \varpi^{n+1} \Lambda_{\O}$.
\end{defn}

Let $\p_f$ be a classical height one prime of $\Tcm$ associated to some ordinary cuspidal eigenform $f$ in the Hida family corresponding to $\mc$.  Write $\Of := \O_f$ for the finite extension of $\Zp$ equal to the integral closure of $\Tcm/\p_f$, and write $\varpi$ for a uniformizer of $\Of$.    We write $\Leispsi(\p_f)$ for the image of $\Leispsi$ in $\Tcm/\p_f \subseteq \Of$.  

For each $a$ with $0 \leq a < p-1$ and $R$ any ring, we have a map $R[[\Zpx]] \to R[[1+p\Zp]]$ given by $[x] \mapsto \omega^a(x) [x/\omega(x)]$ where $x \in \Zpx$.
We write $L^+_p(f,\omega^a)$ and $L^+_p(\m,\omega^a)$ for the respective images of $L_p^+(f)$ and $L^+_p(\m)$ under the corresponding maps.  Further, we write $\mu(f,\omega^a)$ for $\mu(L_p^+(f,\omega^a))$.

\begin{thm}
\label{thm:anlowerbndgen}
Let $\m \subseteq \T$ be a maximal ideal satisfying \all.  Let $\Leispsi$ denote a generator of the Eisenstein ideal $\Iceis \subseteq \Tcm$.  Then
$$
\Leispsi  \text{~~divides~~} L^+_p(\mc)
$$
in $\Tcm[[\Gammac]]$.  In particular,
$$
\mu(f,\omega^a) \geq \ord_\varpi(\Leispsi(\p_f))
$$
for all even $a$ with $0 \leq a \leq p-2$, and all $f$ in the Hida family corresponding to $\mc$.
\end{thm}

\begin{proof}
By Theorem \ref{thm:eisenvanish},  $L^+_p(\m) \in \Ieis_{\m} [[\Zpx]]$.  Projecting to $\Tcm[[\Zpx]]$ implies that $L^+_p(\mc) \in \Iceis_{\mc}[[\Zpx]]$, and is thus divisible by $\Leispsi$ as desired.

The second assertion is immediate from the first as $L^+_p(\mc)$ specializes to $L^+_p(f)$ at the prime $\p_f$ while $\Leispsi$ specializes to the {\it number} $\Leispsi(\p_f) \in \Of$ and thus contributes to the $\mu$-invariant.
\end{proof}

\subsection{Upper bounds}

We now establish an upper bound on $\mu$-invariants of these residually reducible forms under the hypothesis \eqref{m1}.  

\begin{thm}
\label{thm:anupperbndgen}
Fix an Eisenstein maximal ideal $\m \subseteq \T$ satisfying \eqref{m1}.    
Let $f$ be a classical cuspidal eigenform in the Hida family for $\mc$.  Then
$$
\mu(f) \leq \ord_\varpi(a_p(f)-1).
$$
\end{thm}

\begin{proof}
To prove this theorem, we simply check that the valuation of $L^+_p(f)$ evaluated at the trivial character is bounded by $\ord_\varpi(a_p(f)-1)$ as this immediately gives the desired bound on the $\mu$-invariant.
By the interpolation property of $L^+_p(f)$, we have
$$
L^+_p(f)({\bf 1}) = \left( 1 - \frac{1}{a_p(f)} \right) \cdot \varphi^+_{f}(\{ \infty \} - \{0\})(1).
$$
We claim that $\varphi^+_{f}(\{ \infty \} - \{0\})(1)$ is a $p$-adic unit which clearly implies the theorem.

To see this, let $k$ denote the weight of $f$.  Then, by \cite[Proposition 2.5(iii)]{BD-evil}, there is a unique (up to scaling) boundary eigensymbol $\varphi_{k,\psi}$ in $H^1_c(\Gamma_0,\cP_{k-2}^\vee)^+$ with the same system of Hecke-eigenvalues as $E_{k,\psi}^{(p)}$.  We fix $\varphi_{k,\psi}$ to be as defined in \cite[(22)]{BD-evil} which is normalized so that its image in $H^1_c(\Gamma_0,\cP_{k-2}^\vee \otimes \Fp)^+$ is non-zero.

Since $\varphi^+_f$ and $\varphi_{k,\psi}$ have congruent systems of Hecke-eigenvalues, by \eqref{m1}, we have that the images of these symbols in $H^1_c(\Gamma_0,\cP_{k-2}^\vee \otimes \Fp)^+$
are non-zero multiples of one another.  Moreover, the explicit description of $\varphi_{k,\psi}$ in \cite[(22)]{BD-evil}, tells use that $\varphi_{k,\psi}$ is not supported on the $\infty$ cusp while $\varphi_{k,\psi}(\{0\})$ is the functional $P(z) \mapsto P(0)$.  
Thus, 
$$
\varphi_{k,\psi}(\{ \infty \} - \{0\})(1) = 0 - \varphi_{k,\psi}(\{0\})(1) = -1,
$$
and thus  $\varphi^+_{f}(\{ \infty \} - \{0\})(1)$ is a $p$-adic unit as desired.
\end{proof}

\subsection{Conjecture on $\mu$-invariants}

Under the assumption \all~(which implies \eqref{m1}), Theorems \ref{thm:anlowerbndgen} and \ref{thm:anupperbndgen} imply the following string of inequalities:
\begin{equation}
\label{eqn:ineqgen}
\ord_\varpi(\Leispsi(\p_f)) \leq \mu(f) \leq \ord_\varpi(a_p(f)-1).\footnote{Since $U_p - 1 \in \Iceis$, we have $\Leispsi$ divides $U_p-1$ which is consistent with these inequalities.}
\end{equation}

Given the philosophy that $\mu$-invariants are ``as small as they can be", it's tempting to make the following conjecture. 

\begin{conj}
\label{conj:anmu}
Fix a maximal ideal $\m \subseteq \T$ satisfying \all~and let $\Leispsi$ denote a generator of $\Iceis_{\mc}$. Then
\begin{equation}
\label{eqn:conj}
\mu(f,\omega^a) = \ord_\varpi(\Leispsi(\p_f))
\end{equation}
for all even $a$ and all classical $\p_f$ contained in $\mc$. 
\end{conj}

The following proposition shows that to check this conjecture it suffices to check (\ref{eqn:conj}) holds for a single form $f$ in the family.

\begin{prop}
\label{prop:verify}
Assume $\m \subseteq \T$ is a maximal ideal satisfying \all.  For each fixed even number $a$, we have $\mu(f,\omega^a) = \ord_\varpi(\Leispsi(\p_f))$ holds for one classical $\p_f \subseteq \mc$ if and only if this equation holds for all such $\p_f$.
\end{prop}

\begin{proof}
By Theorem \ref{thm:anlowerbndgen}, $L_p^+(\mc,\omega^a)/\Leispsi$ is integral, and by assumption, 
this power series has at least one specialization with vanishing $\mu$-invariant.  Thus, by \cite[Proposition 3.7.3]{EPW}, 
every specialization has vanishing $\mu$-invaraint, proving the proposition.  
\end{proof}

\subsection{When the bounds meet}

In the case when the lower and upper bounds of (\ref{eqn:ineqgen}) meet (e.g.\ when $U_p-1$ generates $\Iceis_{\mc}$), the Iwasawa theory of our situation becomes much simpler.  We begin with a lemma.

\begin{prop}
\label{prop:goren}
Let $\m \subseteq \T$ be an Eisenstein maximal ideal.
The following are equivalent:
\begin{enumerate}
\item \label{part:goren1} $\Iceis_{\mc}$ is generated by $U_p-1$;
\item $\Iceis_{k,\mc}$ is generated by $U_p-1$ for every $k \equiv j(\m) \pmod{p-1}$, $k\geq 2$;
\item \label{part:goren3} $\I^c_{k,\mc}$ is generated by $U_p-1$ for some $k \equiv j(\m) \pmod{p-1}$, $k\geq 2$.
\end{enumerate}
\end{prop}

\begin{proof}
We check that (\ref{part:goren3}) implies (\ref{part:goren1}).  Since $\I^c_{k,\mc}$ is principal, by Theorem \ref{thm:eisenprincipal}, we have $\Iceis_{\mc}$ is principal, generated by say $\Leispsi$.  
Consider now $(U_p-1) / \Leispsi \in \Tc_{\mc}$.  This element specializes to unit in weight $k$, and is thus itself a unit.  Hence, $\Leispsi$ differs from $U_p-1$ by a unit, and thus $U_p-1$ generates $\Iceis_{\mc}$.
\end{proof}

\begin{defn}
An Eisenstein maximal ideal $\m$ of $\T$ satisfies \lu~if any of the equivalent conditions of Proposition \ref{prop:goren} hold.  
\end{defn}

\begin{lemma}
\label{lemma:lutogoren}
\lu~implies \eqref{goren} and \eqref{cuspgoren}.
\end{lemma}

\begin{proof}
This claim follows immediately from Theorem \ref{thm:eisenprincipal}.
\end{proof}

In what follows, we write $\lambda(f)$ for $\lambda(f,\omega^0)$ and $\mu(f)$ for $\mu(f,\omega^0)$.

\begin{thm}
\label{thm:anfullgen}
For an Eisenstein maximal ideal $\m \subseteq \T$ satisfying \fk~and \lu, we have
$$
L^+_{p}(\mc,\omega^0) = \Leispsi \cdot U
$$ 
where $U$ is a unit in $\Tcm[[1+p\Zp]]$.
In particular, for every $f$ in the Hida family of $\mc$, we have
$$
\lambda(f)=0 \text{~~and~~}\mu(f) = \ord_\varpi(a_p(f)-1).
$$  
That is, $L_p(f,\omega^0)$ is simply a power of $p$, up to a unit.
\end{thm}

\begin{proof}
By Lemma \ref{lemma:lutogoren}, \eqref{goren} and \eqref{cuspgoren} hold automatically.  Further,
for $\p_f \subseteq \mc$, \lu~and the inequalities in (\ref{eqn:ineqgen}) tell us the value of the $\mu$-invariant exactly:
$$
\ord_\varpi(\Leispsi(\p_f)) = \mu(f) = \ord_\varpi(a_p(f)-1).
$$
But, as in the proof of Theorem \ref{thm:anupperbndgen}, we know that $L^+_p(f)$ evaluated at ${\bf 1}$ has $p$-adic valuation equal to $\ord_\varpi(a_p(f)-1)$.  Since this valuation also equals the $\mu$-invariant, we get that $\lambda(f)=0$.

By Theorem \ref{thm:anlowerbndgen}, the quotient $L_p(\mc,\omega^0)/\Leispsi$ is integral.  Further,
$$
\frac{L_p(\mc,\omega^0)}{\Leispsi}(\p_f)
= \frac{L_p(f,\omega^0)}{\Leispsi(\p_f)}
$$
has vanishing $\mu$ and $\lambda$-invariants (since $\mu(f) = \ord_\varpi(\Leispsi(\p_f))$), and is thus a unit.  But then $L_p(\mc,\omega^0)/\Leispsi$ has a unit specialization, and it must itself be a unit.
\end{proof}

\subsection{The special case of \eqref{r1}}

In the special case where \eqref{r1} of the introduction holds, the results of the previous sections both become more concrete (with $\Leispsi$ being replaced by a $p$-adic $L$-function) and no longer require any Gorenstein hypotheses (as they are automatically satisfied).  We now state this condition and results more precisely.

\begin{prop} 
\label{prop:r1}
Let $\m \subseteq \T$ be a maximal ideal.
The following are equivalent:
\begin{enumerate}
\item $\dim_{\Lambda} \Tcm = 1$;
\item $\dim_{\Zp} \Tckm = 1$ for some $k \equiv j(\m) \pmod{p-1}$, $k \geq 2$;
\item $\dim_{\Zp} \Tckm = 1$ for all $k \equiv j(\m) \pmod{p-1}$, $k \geq 2$.
\end{enumerate}
\end{prop}

\begin{proof}
This proposition follows from Hida's control theorem: $\Tcm/\p_k \Tcm \cong \Tckm$.
\end{proof}

\begin{defn}
For a maximal ideal $\m \subseteq \T$, we say $\m$ satisfies \eqref{r1} if any of the equivalent conditions of Proposition \ref{prop:r1} hold.  
\end{defn}

\begin{lemma}
\label{lemma:r1togoren}
If an Eisenstein maximal ideal $\m \subseteq \T$ satisfies \fk~and \eqref{r1}, then it also satisfies \eqref{goren} and \eqref{cuspgoren}.
\end{lemma}

\begin{proof}
If $(p,j)$ satisfies \eqref{r1}, then $\Tckm$ is rank 1 over $\Zp$ and clearly \eqref{cuspgoren} holds.  Further, by Lemma \ref{lemma:unique_eisen}, \fk~implies that $\Tkm$ has rank 2 over $\Zp$.  This in turn implies that $\Tkm$ is generated by a single element over $\Zp$.  In particular, $\Tkm$ is a complete intersection and thus \eqref{goren} holds.
\end{proof}

By the above lemma and Theorem \ref{thm:eisenprincipal}, we know that $\Iceis_{\mc}$ is principal when \eqref{r1} holds.  In this case, work of Wiles and Ohta \cite{Wiles-MC,Ohta-congruence} imply that this generator can be taken to be a $p$-adic $L$-function.  More precisely, let $L_p(\psi,\kappa)$ be defined by as in \cite{BD-evil}.  Here we view $L_p(\psi,\kappa)$ as a function on weight space and we have:
$$
L_p(\psi,z^k) = -(1-\psi^{-1}(p)p^{k-1}) \frac{B_{k,\psi^{-1}}}{k}
$$
for $k \geq 1$ with $(-1)^k = \psi(-1)$.  Set $\Gammawj = 1+p\Zp$ and view $\Zp[[\Gammawj]]$ as Iwasawa functions on $\W_j$.   Then $L_p(\psi,\kappa)|_{\W_j}$ is in $\Zp[[\Gammawj]]$ as long as either $\psi$ is non-trivial or $j \neq 0$.

Note that under \eqref{r1} we can identify $\Tcm$ with $\Zp[[\Gammawj]]$ and we can and will view $L_p(\psi,\kappa)|_{\W_j}$ as an element of $\Tcm$.

\begin{thm}
\label{thm:eisgen}
If $\m \subseteq \T$ satisfies \eqref{r1}, then $\Iceis_{\mc}$ is generated by $L_p(\psi_{\m}^{-1},\kappa)|_{\W_{j(\m)}}$.
\end{thm}

\begin{proof}
This theorem is simply \cite[Proposition 3.1.9]{Ohta-congruence} in the special case when \eqref{r1} holds.
\end{proof}

Under \eqref{r1}, we can consider  $L_p^+(\mc)$  in $\Tcm[[\Zpx]] \cong \Zp[[\Gammawj \times \Zpx]]$, and we write $L_p^+(\mc)$ as $L_p^+(\kappa)$ where $\kappa$ is thought to range over weights in $\W_j$.

The following theorem is the simplified verison of Theorem \ref{thm:anlowerbndgen} when \eqref{r1} holds.  We denote by $f_k$ the unique normalized cuspidal eigenform of weight $k$ and level $\Gamma_0$ congruent to $E_{k,\psi}^{\ord}$.

\begin{thm}
\label{thm:anlowerbnd}
Fix an Eisenstein maximal ideal $\m \subseteq \T$ satisfying  \fk~and \eqref{r1}.  Then
$$
L_p(\psi_{\m}^{-1},\kappa)\text{~~divides~~} L^+_{p}(\kappa)
$$
in $\Zp[[\Gammawj \times \Zpx]]$.  In particular,
$$
\mu(f_k,\omega^a) \geq \ord_p(L_p(\psi_{\m}^{-1},z^k)) = \ord_p( B_{k,\psi_{\m}}/k)
$$
for each even $a$ with $0 \leq a \leq p-1$ and for all $k \geq 2$ with $k \equiv j(\m) \pmod{p-1}$.
\end{thm}

\begin{proof}
This theorem follows immediately from Theorem \ref{thm:anlowerbndgen} and Theorem \ref{thm:eisgen} noting that Lemma \ref{lemma:r1togoren} gives us that \eqref{goren} and \eqref{cuspgoren} hold.
\end{proof}

We can likewise deduce the analogue of Theorem \ref{thm:anfullintro} of the introduction from Theorem \ref{thm:anfullgen}.

\begin{thm}
Fix an Eisenstein maximal ideal $\m \subseteq \T$ satisfying \fk~and \eqref{r1} and
for which
$$
\ord_p(B_{k,\psi_{\m}}/k) = \ord_p(a_p(k)-1).
$$
for some $k \equiv j(\m) \pmod{p-1}$.
Then 
$$
L^+_{p}(\kappa,\omega^0) = L_p(\psi_{\m}^{-1},\kappa) \cdot U
$$ 
where $U$ is a unit in $\Zp[[\Gammawj \times (1+p\Zp)]$.
In particular, for every $k \geq 2$ with $k \equiv j(\m) \pmod{p-1}$, we have that 
$$\lambda(f_k)=0 \text{~~and~~}\mu(f_k) = \ord_p(L_p(\psi_{\m},k)) = \ord_p(B_{k,\psi_{\m}}/k).$$  
That is, $L_p^+(f_k,\omega^0)$ is simply a power of $p$, up to a unit.
\end{thm}

\subsection{Numerical verification of Conjecture \ref{conj:anmu}}
\label{sec:numerical}
Using Proposition \ref{prop:verify}, we  numerically verified that the $a=0$ case of this  conjecture holds for $N=1$ and all irregular pairs $(p,k)$ with $p<2000$. However, in every such case, $\ord_\varpi(\Leispsi(\p_f)) =  \ord_\varpi(a_p(f)-1)$, and thus
the inequalities in (\ref{eqn:ineqgen}) automatically imply the conjecture holds.
Thus these checks do not yield much evidence for the conjecture.

However, when $a \not\equiv 0 \pmod{p-1}$,  the upper bound in (\ref{eqn:ineqgen}) does not a priori hold and we instead directly computed the relevant $\mu$-invariant to verify this conjecture.  Namely, we verified Conjecture \ref{conj:anmu} holds when $N=1$ and $p < 750$ for all even values of $a$ (except when $p=547$ and $k=486$ when $\eqref{r1}$ does not hold).  Further, for $N$ prime, let $\chi_N$ denote the unique quadratic character of conductor $N$.  We verified this conjecture for $\chi_N$ for $N=3,5,7$ and $p< 500,400,300$ respectively, again for all even values of $a$.  These computations included approximately 100 distinct eigenforms and in all of these examples the above mentioned exception was the only time \eqref{r1} failed.

We further mention that in all of  these tests there was a single example where the $a=0$ case of the conjecture did not immediately follow from the bounds in (\ref{eqn:ineqgen}).  Namely, $\ord_{19}(B_{8,\chi_5}) = 2$ and in this case, there is a unique cuspidal eigenform $f \in S_8(\Gamma_1(5),\chi_5,\Z_{19})$ congruent to $E_{8,\chi_5}$, and as expected from the Bernoulli divisibility, these forms are congruent modulo $19^2$.  Moreover,  $\ord_{19}(a_{19}(f) - 1) = 3$ and thus the bounds in (\ref{eqn:ineqgen}) do not meet. 
In this case, we computed that $\mu(f,\omega^0) = 2$ which verifies Conjecture \ref{conj:anmu}. 
We further computed that $\lambda(f,\omega^0) = 1$ in contrast to the behavior of these invariants described by Theorem \ref{thm:anfullgen} when the bounds of (\ref{eqn:ineqgen}) do meet.

We briefly mention here the method used to compute the relevant $\mu$-invariants as it was a bit novel and to explain why $\eqref{r1}$ was needed.  Namely, fix $N$, $\psi$ and a pair $(p,k)$ such that $p \mid B_{k,\psi}$.  We formed a random overconvergent modular symbol $\Phi$ of level $Np$, weight $k$, and nebentype $\psi$ as in \cite{PS-explicit}.  Working modulo a low accuracy (around modulo $p^4$), we projected $\Phi$ to the ordinary subspace (i.e.\ computed $\Phi | U_p^4$), and then used the other Hecke operators to form a (non-boundary) eigensymbol whose $\ell$-th eigenvalue was congruent to $1+\chi(\ell) \ell^{k-1}$.  When $\eqref{r1}$ holds, there is a unique such symbol and thus this symbol must be the one associated to the cuspidal eigenform we are considering.  Evaluating the resulting symbol at $\{\infty\}-\{0\}$ then gives the relevant $p$-adic $L$-function (modulo $p^4$) from which an upper bound for the $\mu$-invariant can be obtained.  In every case, this upper bound matched our proven lower bound verifying the conjecture.

\section{Selmer groups}

Let $f = \sum_n a_n q^n$ be a normalized eigenform in $S_k(\Gamma_1(N), \psi, \Qpbar)$.  We further assume that $f$ is $p$-ordinary, that is, $\ord_p(a_p) =1$.
 Attached to $f$ we have its (homological) Galois representation $\rho_f : G_{\Q} \to \Aut(V_f)$ where $V_f$ is a two-dimensional vector space over $\Kf = \Qp(\{a_n\}_n)$.   We note that $\det(\rho_f) = \psi \ve^{k-1}$ where $\ve$ is the $p$-adic cyclotomic character.  Set $\Of$ equal to the ring of integers of $\Kf$, and choose a Galois stable $\Of$-lattice $T_f \subseteq V_f$.
 
Since $f$ is $p$-ordinary, the representation $\rho_f$ is locally reducible at $p$ \cite[Theorem 4.2.7(2)]{Hida-GMF}, and we have an exact sequence of $G_{\Qp}$-representations
$$
0 \to \Of(\eta^{-1} \psi \ve^{k-1}) \to T_f \to \Of(\eta) \to 0
$$ 
where $\eta$ is an unramified character which sends $\Frob_p$ to $\alpha_p$, the unit root of $x^2-a_p x + \psi(p) p^{k-1}$.

Further, set $A_f = V_f/T_f$ which is isomorphic to $(\Kf / \Of)^2$ and endowed with an action of $G_{\Q}$.  
We then have an exact sequence of $G_{\Qp}$-representations
\begin{equation}
\label{eqn:ordline}
0 \to \KOf(\eta^{-1} \psi \ve^{k-1}) \to A_f \to \KOf(\eta) \to 0.
\end{equation}
and
\begin{equation}
\label{eqn:ordlinetor}
0 \to \Of / \varpi^r \Of(\eta^{-1}\psi \ve^{k-1}) \to A_f[\varpi^r] \to \Of / \varpi^r \Of(\eta) \to 0
\end{equation}
where  $\varpi$ is a uniformizer of $\Of$.  Lastly, set $\Ff = \Of/\varpi \Of$.

\subsection{Definitions of Selmer groups}

Let $\Qinf$ denote the cyclotomic $\Zp$-extension of $\Q$ and set $\Gamma := \Gal(\Qinf/\Q)$.
Following Greenberg \cite{Greenberg-ordinary}, we define the Selmer group $\Sel(\Qinf,A_f)$ as the collection of classes $\sigma$ in $H^1(\Qinf,A_f)$ such that (1) for $v$ a place of $\Qinf$ not over $p$, $\res_v(\sigma)$ is unramified; that is, $\sigma$ is in the kernel of
$$
H^1(\Qinf,A_f) \to H^1(I_v,A_f)
$$
where $I_v$ is a choice of the inertia group at $v$, and (2) $\sigma$ lies in the kernel of 
$$
H^1(\Qinf,A_f) \to H^1(I_{p,\infty},A_f) \to H^1(I_{p,\infty},K/\O)
$$
where the final map is induced by (\ref{eqn:ordline}) and $I_{p,\infty}$ is a choice of an inertia group at the unique prime of $\Qinf$ over $p$.

One can as well analogously define these Selmer groups over $\Q$ or over $\Q(\mu_{p^\infty})$.  We can also define Selmer groups of $A_{f,j} := A_f \otimes \omega^j$ by simply twisting (\ref{eqn:ordline}) and using this new sequence to define the local condition at $p$.  Since $\Q(\mu_p)/\Q$ has degree prime to $p$, we have the following relationship between these Selmer groups:
$$
\Sel(\Qcyc,A_f) \cong \bigoplus_{j=0}^{p-2} \Sel(\Qinf,A_{f,j}).
$$

\begin{remark}
We note that if $f$ arises from some elliptic curve $E/\Q$, then $\Sel(\Q,A_f)$ defined above may not equal the $p$-adic Selmer group of $E$, $\Sel(\Q,E[p^\infty])$, whose local conditions are defined by the Kummer map.  Indeed, for $\ell \neq p$, the image of the Kummer map vanishes while there can be unramified cocyles at $\ell$ for primes of bad reduction.  As a result, $\Sel(\Q,A_f)$ can be larger than $\Sel(\Q,E[p^\infty])$ and we will see that control theorems (see Proposition \ref{prop:controlinf}) work better for these larger Selmer groups.
\end{remark}

The definition of the local condition at $p$ in all of these Selmer groups is given by restriction to the inertia group at $p$.  However, as proven in the following lemma, restricting to the decomposition at $p$ instead does not change the definition of these Selmer groups.  (In the language of Greenberg \cite{Greenberg-ordinary}, the strict Selmer group matches the full Selmer group in our case.)

\begin{lemma}
\label{lemma:strict}
Let $D_{p,\infty} \supset I_{p,\infty}$ be a choice of decomposition and inertia groups at the unique prime of $\Qinf$ over $p$, and likewise define $D_p \supset I_p$ as decomposition and inertia groups at $p$.  Then the following maps are injective:
$$
H^1(D_{p},\KOf(\eta\omega^j)) \to H^1(I_{p},\Kf/\Of(\omega^j)).
$$
and
$$
H^1(D_{p,\infty},\KOf(\eta\omega^j)) \to H^1(I_{p,\infty},\Kf/\Of(\omega^j)).
$$
\end{lemma}

\begin{proof}
We check the injectivity of the first map as the argument for the second is identical.  The kernel of this map is 
$H^1(\Qp^{\un}/\Qp,(\KOf(\eta\omega^j))^{I_p})$ where $\Qp^{\un}$ is the maximal unramified extension of $\Qp$.  If $j \not\equiv 0 \pmod{p-1}$, then $(\KOf(\eta\omega^j))^{I_p}=0$ and we are done.

Otherwise, since $\Gal(\Qp^{\un}/\Qp)$ is topologically cyclic generated by $\Frob_p$, we have that $H^1(\Qp^{\un}/\Qp,\KOf(\eta))$ is given by the cokernel of
$$
\KOf(\eta) \stackrel{\Frob_p - 1}{\lra} \KOf(\eta).
$$
As long as $\Frob_p -1$ is not identically zero on $\KOf(\eta)$ we are done since $\KOf$ is divisible.  Since $\eta(\Frob_p) = \alpha_p$, the unit root of $x^2 - a_p x + \psi(p) p^{k-1}$, we simply need to check then that $\alpha_p \neq 1$.  But if $\alpha_p = 1$, then the other root of this quadratic would be $\psi(p) p^{k-1}$.  Hence $a_p = 1 + \psi(p) p^{k-1}$ which violates the Weil bounds.
\end{proof}

We will need to consider one last kind of Selmer group; namely, we will define the Selmer group of the finite Galois module $A_f[\varpi^n]$.  We follow the method of \cite[Section 4.2]{EPW} and note that these Selmer groups depend not just on $A_f[\varpi^n]$, but on the modular form $f$ itself.  Namely, if we are working over $\Qinf$, then our local condition at $v \nmid p$ is given by the image of $A_f^{G_v}/\varpi^n$ in $H^1(\Q_{\infty,v},A_f[\varpi^n])$ where $G_v = G_{\Q_{\infty,v}}$. At $p$, we use (\ref{eqn:ordlinetor}) to define our local condition just as in the characteristic 0 case.  As before, we can also define these objects over $\Q$ or $\Q(\mu_{p^\infty})$ and we can analogously define Selmer groups of $A_{f,j}[\varpi^n]$.

The reason for the above definition of the local condition at $v \nmid p$ is that it is exactly the condition needed to make the following lemma true.

\begin{lemma}
\label{lemma:selmodp}
There are natural maps
$$
\Sel(\Q,\Afj[\varpi^r]) \to \Sel(\Q,\Afj)[\varpi^r],
$$
$$
\Sel(\Qinf,\Afj[\varpi^r]) \to \Sel(\Qinf,\Afj)[\varpi^r]
$$
which are surjective and have respective kernels $H^0(\Q,\Afj)/p^rH^0(\Q,\Afj)$ and $H^0(\Qinf,\Afj)/p^rH^0(\Qinf,\Afj)$.
\end{lemma}

\begin{proof}
Verifying this lemma is just a diagram chase.
\end{proof}

\subsection{The control theorem}

We now state a key control theorem for these Selmer groups following closely \cite{Greenberg-CIME}.

 \begin{prop}
\label{prop:controlinf}
If $H^0(\Q,\Afj[\varpi])=0$, then the natural map
$$
\Sel(\Q,\Afj) \to \Sel(\Qinf,\Afj)^\Gamma
$$
is injective and has cokernel bounded by the size of $\Of/(\alpha_p-1)\Of$.  Moreover, the map is surjective if $j \not\equiv 0 \pmod{p-1}$.
\end{prop}

\begin{proof}
By Lemma \ref{lemma:strict}, we have a commutative diagram:
$$
\begin{tikzcd}
\Sel(\Q,\Afj) \arrow[hookrightarrow]{r}{} \arrow{d}{} &H^1(\Q_\Sigma/\Q,\Afj) \arrow{r}{} \arrow{d}{h} & \H^{(p)}_{\loc} \times H^1(D_p,\KOf(\eta \omega^j)) \arrow{d}{r}\\
\Sel(\Qinf,\Afj)^\Gamma \arrow[hookrightarrow]{r}{}  &H^1(\Q_\Sigma/\Qinf,\Afj)^\Gamma \arrow{r}{} & (\H^{(p)}_{\infty,\loc} \times H^1(D_{p,\infty},\KOf(\eta \omega^j)))^\Gamma
\end{tikzcd}
$$
where $\H^{(p)}_{\loc} := \oplus_{\ell | N} H^1(I_\ell,A_f)$ and $\H^{(p)}_{\infty,\loc} := \oplus_{v | N} H^1(I_v,A_f)$ where $v$ runs over primes of $\Qinf$.   We seek to apply the snake lemma and thus we must analyze the kernel and cokernel of $h$ and the kernel of $r$.

The cokernel of $h$ maps to $H^2(\Gamma,H^0(\Qinf,\Afj))$ which vanishes since $\Gamma \cong \Zp$ has cohomological dimension 1.  The kernel of $h$ equals $H^1(\Gamma,H^0(\Qinf,\Afj))$.  To see that this group vanishes, it suffices to see that $H^0(\Qinf,\Afj)$ vanishes.  But by assumption $H^0(\Q,\Afj[\varpi])=0$ which implies $H^0(\Q,\Afj)=0$ which in turn implies $H^0(\Qinf,\Afj) =0$ as $\Gal(\Qinf/\Q)$ is pro-$p$.  Thus, $h$ is an isomorphism.

To determine the kernel of $r$, we first note that $I_\ell \cong I_v$ if $v | \ell \neq p$ as $\Qinf/\Q$ is unramified at $\ell$.  Thus $\H^{(p)}_{\loc}$ injects into $\H^{(p)}_{\infty,\loc}$, and
$\ker(r)$ equals the kernel of $H^1(D_p,\KOf(\eta \omega^j)) \to H^1(D_{p,\infty},\KOf(\eta \omega^j))$ which in turn is isomorphic to 
$$
H^1(\Q_{\infty,p}/\Qp, H^0(\Q_{\infty,p},\KOf(\eta \omega^j))).
$$  
If $j \not\equiv 0 \pmod{p-1}$, then the above $H^0$-term vanishes and $\ker(r)=0$.  Otherwise, since $\Gal(\Q_{\infty,p}/\Qp)$ is topologically cyclic,
we have $\ker(r)$ is given by the $\Gal(\Q_{\infty,p}/\Qp)$-coinvariants of $H^0(\Q_{\infty,p},\KOf(\eta))$.  
Since $\eta$ is the unramified character sending $\Frob_p$ to $\alpha_p$, we have 
$$
H^0(\Q_{\infty,p},\KOf(\eta)) = \ker(\KOf \stackrel{\times (\alpha_p-1)}{\lra} \KOf)
\cong \Of/(\alpha_p-1)\Of.
$$ 
But then 
$$
| \ker(r) | = | H^0(\Q_{\infty,p},\KOf(\eta)) | = | \Of/(\alpha_p-1)\Of |.
$$
The proposition then follows, in either case, from the snake lemma.
\end{proof}

\subsection{The main conjecture}

Let $\Lambda_{\O} := \Of[[1+p\Z_p]] \cong \Of[[\Gal(\Qinf/\Q)]]$. The following main conjecture relates these Selmer group over $\Qinf$ to $p$-adic $L$-functions.

\begin{conj}[Main Conjecture]
We have $\Sel(\Qinf,\Afj)^\vee$ is a finitely generated torsion $\Lambda_{\O}$-module, and 
$$
\chr_{\Lambda_{\O}[1/p]} \Sel(\Qinf,\Afj)^\vee = L^+_p(f,\omega^j) \cdot \Lambda_{\O}[1/p].
$$
\end{conj}

In this level of generality, where the Galois representation of $f$ is not assumed to be residually irreducible, we can only state a main conjecture with $p$ inverted.  The reason for this is  that $L_p(f)$ only depends upon the modular form $f$ while $\Sel(\Qinf,\Afj)$ depends on a choice of a lattice in the Galois representation $V_f$.  The choice of a lattice can change the left hand side by powers of $\varpi$.  The issue of which lattice to pick to correctly match the $p$-adic $L$-function will be further discussed in the next section.

The following is a deep theorem of Kato which proves half of the main conjecture.

\begin{thm}[Kato]
\label{thm:katomc}
We have that  $\Sel(\Qinf,\Afj)^\vee$ is a finitely generated torsion $\Lambda_{\O}$-module, and 
$$
\chr_{\Lambda_{\O}[1/p]} \Sel(\Qinf,\Afj)^\vee \text{~~divides~~} L^+_p(f,\omega^j) 
$$
in $ \Lambda_{\O}[1/p]$.
\end{thm}

\begin{proof}
See \cite[Theorem 17.4.1 and 17.4.2]{Kato}.  Note that the hypothesis that $\rhobar_f$ has large image is not used in this part of Kato's work.
\end{proof}

\section{Algebraic results}
\label{sec:algebraic}

We continue with the notation of the previous section so that $f$ is a normalized cuspidal eigenform in $S_k(\Gamma_1(N),\psi,\Qpbar)$ where $p\nmid N$ and $p \nmid \varphi(N)$.  

\begin{defn}
Let $\eis(f)$ denote the largest integer $n \geq 0$ such that $f \equiv E^{\ord}_{k,\psi} \pmod{\varpi^n}$ where this congruence takes place in $\Of / \varpi^n \Of[[q]]$.
\end{defn}

Throughout this section we assume that $\eis(f) > 0$.  In particular, $f$ is ordinary and let $\m = \m_f$ denote the maximal ideal in $\T$ corresponding to $f$.  Note then that $\psi = \psi_{\m}$.  We further assume that $\m$ satisfies \fk.  

We note again that there is not a unique lattice (up to homethety) in the Galois representation $V_f$.  The following lemma describes the situation more precisely.

\begin{lemma}
\label{lemma:lattices}
If $\m_f$ satisfies \fk, then in $V_f$ there exists a sequence of Galois stable lattices:
$$
T_0 \subsetneq T_1 \subsetneq \cdots \subsetneq T_{\eis(f)},
$$
no two of which are homethetic, such that 
\begin{enumerate}
\item $T_i / T_{i-1}$ is isomorphic to $\Ff$;

\item For $i \neq 0, \eis(f)$, we have $T_i / \varpi T_i$ is a split extension of $\Ff$ and $\Ff(\psi\omega^{k-1})$.
\item For $i = \eis(f)$, we have a non-split extension
$$
0 \to \Ff(\psi\omega^{k-1}) \to T_{\eis(f)} / \varpi T_{\eis(f)} \to \Ff \to 0.
$$
\item For $i = 0$, we have a non-split extension
$$
0 \to \Ff \to T_0 / \varpi T_0 \to \Ff(\psi\omega^{k-1}) \to 0.
$$
\item $T_i / \varpi^i T_i$ contains a submodule isomorphic to $\Of/\varpi^i \Of(\psi\ve^{k-1})$;
\end{enumerate}
\end{lemma}

\begin{proof}
As $f$ admits a congruence to $E_{k,\psi}^{\ord}$ modulo $\varpi^{\eis(f)}$, the existence of such a chain of lattices is standard as in \cite[section 1.2]{B-Hawaii}. If a longer chain existed, then $f$ would admit a congruence modulo $\varpi^{m}$ with $m>\eis(f)$ to some Eisenstein series (necessarily not $E_{k,\psi}^{\ord}$).  However, \fk~and Lemma \ref{lemma:unique_eisen} prevent this possibility. 
\end{proof}

We will see that the choice of lattice in $V_f$ will affect the value of the $\mu$-invariant of the corresponding Selmer group.

\subsection{Lower bounds}

Set $A_{f,j}^{(r)} := V_{f}/T_r \otimes \omega^j$ for $0 \leq r \leq \eis(f)$.  
By Theorem \ref{thm:katomc}, the Selmer group $\Sel(\Qinf,A_{f,j}^{(r)})$ is a cotorsion $\Lambda_{\O}$-module and thus has associated $\mu$ and $\lambda$-invariants.  
Following Greenberg, we now give lower bounds on these $\mu$-invariants which grow as $r$ grows.

\begin{thm}
\label{thm:alglowerbndgen}
For $0 \leq r \leq \eis(f)$ and for $j$ even, we have 
$$
\mu(\Sel(\Qinf,A_{f,j}^{(r)})^\vee) \geq r.
$$
\end{thm}

\begin{proof}
By Theorem \ref{lemma:lattices}, we have that $A_{f,j}^{(r)}$ contains a submodule isomorphic to $\Of/\varpi^r \Of(\psi\ve^{k-1}\omega^j)$ which is cyclic, odd, and ramified at $p$.  Our theorem then follows exactly as in \cite[Proposition 5.7]{Greenberg-CIME} where an analogous statement is proven in the case of elliptic curves. 
\end{proof}

To relate this discussion to the our bounds on analytic $\mu$-invariants, especially Theorem \ref{thm:anlowerbndgen}, we have the following lemma.

\begin{lemma}
\label{lemma:Leis}
Let $\m \subseteq \T$ be a maximal ideal and assume that the cuspidal Eisenstein ideal $\Iceis_{\mc} \subseteq \Tcm$ is principal with generator $\Leispsi$.  Then for $\p_f \subseteq \mc$, we have
$$
\eis(f) = \ord_{\varpi}( \Leispsi(\p_f)).
$$
\end{lemma}

\begin{proof}
This lemma follows simply from the definition of $\I$.
\end{proof}

\subsection{Upper bounds}
Let $\m$ be a maximal ideal of $\T$, and let $f$ be a classical eigenform in the Hida family for $\mc$.  Set $T = T_{\eis(f)}$ as in Lemma \ref{lemma:lattices}, and $A_f := V_f / T$.  This is the choice of lattice which maximizes our lower bound on the $\mu$-invaraint.  This section will be devoted to proving an upper bound on the $\mu$-invariant of $\Sel(\Qinf,A_f)^\vee$. Compare with Theorem \ref{thm:anupperbndgen} on the analytic side.

Consider now our nebentype character $\psi: (\Z/N\Z)^\times \to \Qpbar$ as a Galois character, and let $K_\psi \subseteq \Q(\mu_N)$ be the smallest field which trivializes $\psi$.  We note that $K_\psi$ is disjoint from $\Q(\mu_p)$, and thus $\Gal(K_\psi(\mu_p)/K_\psi) \cong \Delta$.  In particular, we can discuss the $\omega^j$-eigenspaces in $\Cl(K_\psi(\mu_p))$.  We need a Vandiver-type hypothesis to achieve an upper bound on the $\mu$-invaraint.
Namely, 
\begin{equation}
\label{vandpsi}
\tag{$\text{Vand~}\psi\omega^r $}
\text{the~}\psi\omega^{r}\text{-eigenspace~of~}\Cl(K_\psi(\mu_p))[p]{~vanishes}.
\end{equation}

\begin{remark}
We note that this condition does not hold generally when $N>1$.  See \cite[Corollary 1.4]{Wake-gorenstein} for explicit counter-examples and the relation of these counter-examples to $\Tm$ being Gorenstein.  
\end{remark}

\begin{thm}
\label{thm:algupperbndgen}
Let $\m \subseteq \T$ be an maximal ideal satisfying:
\begin{enumerate}
\item \fk
\item $\vandpsiv{\psi^{-1}}{2-j(\m)}$
\item $j(\m) = 2 \implies \psi(p) \neq 1$.  
\end{enumerate}
Then for $f$ a classical eigenform in the Hida family for $\mc$, we have  
$$
\mu(\Sel(\Qinf,A_f)^\vee) \leq \ord_\varpi(a_p(f)-1).
$$
\end{thm}

We begin with a simple lemma which gives an upper bound on the $\mu$-invariant of a $\Lambda$-module in terms the module's $\Gamma$-coinvariants.

\begin{lemma}
\label{lemma:mu}
Let $X$ be a finitely generated torsion $\Lambda_{\Of}$-module with no finite submodules. If $X_{\Gamma}$ is finite with size bounded by $M$, then $q^{\mu(X)} \leq M$.  
\end{lemma}

\begin{proof}
Using the structure theorem of finitely generated $\Lambda$-modules, we know that we have a map
$$
X \lra Y
$$
with finite kernel and cokernel where $Y = \oplus_i \Lambda / f_i^{n_i} \Lambda$ and the $f_i$ are irreducible.  Since $X$ has no finite submodule, we get an exact sequence
$$
0 \to X \lra Y \lra K \to 0
$$
with $K$ finite.  Thus,
$$
0 \to X^\Gamma \to Y^\Gamma \to K^\Gamma \to X_\Gamma \lra Y_\Gamma  \lra K_\Gamma \to 0.
$$
Since we are assuming that $X_\Gamma$ is finite, we have $f_i(0) \neq 0$ for all $i$, and thus $Y^\Gamma=0$.  Our exact sequence is then just
$$
0 \to K^\Gamma \to X_\Gamma \lra Y_\Gamma \lra K_\Gamma \to 0,
$$
and we deduce that $|X_\Gamma| = |Y_\Gamma|$  since $|K^\Gamma| = |K_\Gamma|$.
Finally, since $(\Lambda_{\O} / \varpi^r \Lambda_{\O})_{\Gamma}$ has size $q^r$, we have  $|Y_\Gamma| \geq q^{\mu(X)}$.  Thus 
$$
M \geq |X_\Gamma| = |Y_\Gamma| \geq q^{\mu(X)}
$$
as desired.
\end{proof}

In light of Lemma \ref{lemma:mu}, we need to bound the size of $\Sel(\Qinf,A_f)^{\Gamma}$.  By Proposition \ref{prop:controlinf}, we thus need to control the size of $\Sel(\Q,A_f)$.  This is done in the following lemma. 

\begin{prop}
\label{prop:vandtoSel0}
Under the hypotheses of Theorem \ref{thm:algupperbndgen}, we have  $\Sel(\Q,A_f)=0$.
\end{prop}

\begin{proof}
By definition of $T:=T_{\eis(f)}$, there exists a non-split short exact sequence
$$
0 \to \Ff(\psi\omega^{k-1}) \to T/\varpi T \to \Ff \to 0,
$$
In particular, $H^0(\Q,A_f[\varpi]) = H^0(\Q,\Ff(\psi\omega^{k-1}))=0$ as $\psi\omega^{k-1}$ is odd.
Thus, by Lemma \ref{lemma:selmodp}, to prove this proposition it suffices to check that $\Sel(\Q,A_f[\varpi])=0$.

The above short exact sequence yields
$$
0 \to H^0(\Q_\Sigma/\Q,\Ff) \to H^1(\Q_\Sigma/\Q,\Ff(\psi\omega^{k-1})) \to H^1(\Q_\Sigma/\Q,A_f[\varpi]) \to H^1(\Q_\Sigma/\Q,\Ff).
$$
We claim that the image of $H^1(\Q_\Sigma/\Q,\Ff(\psi\omega^{k-1}))$ in $H^1(\Q_\Sigma/\Q,A_f[\varpi])$ contains $\Sel(\Q,A_f[\varpi])$.  

To this end, take $\phi \in \Sel(\Q,A_f[\varpi])$ and let $\im(\phi)$ denote the image of $\phi$ in $H^1(\Q_\Sigma/\Q,\Ff) = \Hom(\Gal(\Q_\Sigma/\Q),\Ff)$, and we will check that $\im(\phi)$ vanishes.  Since $\phi$ is a Selmer class, $\phi$ is unramified outside of $p$.  Further, the very definition of the local condition at $p$ tells us that $\im(\phi)$ is  unramified at $p$.  In particular, $\im(\phi)$ is unramified everywhere and hence zero as desired.

Thus
$$
\dim_{\Ff}(\Sel(\Q,A_f[\varpi])) \leq \dim_{\Ff}(H^1(\Q_\Sigma/\Q,\Ff(\psi\omega^{k-1}))) - 1
$$
with the $-1$ coming from $H^0(\Q_\Sigma/\Q,\Ff) \cong \Ff$.  The below lemma (Lemma \ref{lemma:Wiles}) whose hypotheses are satisfied by \fk~and our running assumption that $p \nmid \varphi(N)$, then gives
$$
\dim_{\Ff}(\Sel(\Q,A_f[\varpi])) \leq \dim_{\Fp}(\Cl(K_\psi(\mu_p))[p]^{(\psi^{-1} \omega^{2-k})}) 
$$
which is 0 by $\vandpsiv{\psi^{-1}}{2-k}$.
\end{proof}

\begin{lemma}
\label{lemma:Wiles}
Let $\psi$ be a character of conductor $N$ and let $\Sigma$ denote the set of primes dividing $pN$.  Assume that:\
\begin{enumerate}
\item $p \nmid \varphi(N)$
\item $(-1)^k = \psi(-1)$
\item $k \equiv 2 \pmod{p-1} \implies \psi(p) \neq 1$.
\end{enumerate}
Then
$$
\dim_{\F}(H^1(\Q_\Sigma/\Q,\F(\psi\omega^{k-1}))) \leq \dim_{\F}(\Cl(K_\psi(\mu_p))[p]^{(\psi^{-1} \omega^{2-k})}) + 1.
$$
\end{lemma}

\begin{proof}
Let $X=\F(\psi\omega^{k-1})$ and $X^*=\F(\psi^{-1}\omega^{2-k})$.
Set $H^1_f(\Q_\Sigma/\Q,X^*)$ equal to the subcollection of classes in $H^1(\Q_\Sigma/\Q,X^*)$ which are locally trivial at all places in $\Sigma$.  Then, by \cite[Proposition 1.6]{Wiles-FLT}, we have
$$
\frac{\#H^1(\Q_\Sigma/\Q,X)}{\#H^1_f(\Q_\Sigma/\Q,X^*)} = 
h_\infty \prod_{{\ell} \in \Sigma} h_{\ell},
$$
where 
$$
h_{\ell} = \#H^0(\Q_{\ell},X^*) \text{~~and~~}
h_\infty = \frac{\#H^0(\R,X^*) \cdot \#H^0(\Q,X)}{\#H^0(\Q,X^*)}.
$$

We analyze each term individually.  First note that $H^0(\Q,X) = 0$ as $\psi\omega^{k-1}$ is non-trivial (being odd) and $H^0(\Q,X^*) = 0$ as we have assumed that $\psi^{-1}\omega^{2-k}$ is non-trivial.   Further, $\#H^0(\R,X^*) = q$ as $\psi^{-1} \omega^{2-k}$ is an even character.  Now, for ${\ell} \in \Sigma-\{p\}$, clearly $H^0(\Q_{\ell},X) = 0$ as $\psi$ is ramified at ${\ell}$.  Lastly, $H^0(\Qp,X^*)=0$ as either $k \not\equiv 2 \pmod{p-1}$ or $\psi(p) \neq 1$ by assumption.

Thus, $h_\infty = q$, $h_{\ell} = 1$ for all $\ell \in \Sigma$, and
$$
\#H^1(\Q_\Sigma/\Q,X) = q \cdot \#H^1_f(\Q_\Sigma/\Q,X^*).
$$
Let $\Delta_{\psi} = \Gal(K_\psi(\mu_p)/\Q)$ which has size prime-to-$p$ as we are assuming that $p \nmid \varphi(N)$.  Thus we have 
\begin{align*}
H^1(\Q_\Sigma/\Q,X^*) 
&\stackrel{\sim}{\to} 
H^1(\Q_\Sigma/K_\psi(\mu_p),X^*)^{\Delta_{\psi}}\\
&\cong
\Hom_{\Delta_{\psi}}(\Gal(\Q_\Sigma/K_\psi(\mu_p)),X^*)\\
&\cong
\Hom(\Gal(\Q_\Sigma/K_\psi(\mu_p))^{(\psi^{-1}\omega^{2-k})},\F)
\end{align*}
The image of $H^1_f(\Q_\Sigma/\Q,X^*)$ in $H^1(\Q_\Sigma/\Q,X^*)$ thus lands in 
$$
\Hom(\Gal(H_\psi/K_\psi(\mu_p))^{(\psi^{-1}\omega^{2-k})},\F) 
$$
where $H_{\psi}$ denotes the Hilbert class field of $K_\psi(\mu_p)$.
Hence,
$$
\#H^1(\Q_\Sigma/\Q,\F(\psi\omega^{k-1}))
\leq q \cdot \#\Cl(K_\psi(\mu_p))[p]^{(\psi^{-1}\omega^{2-k})}
$$
which proves the lemma.
\end{proof}

\begin{proof}[Proof of Theorem \ref{thm:algupperbndgen}]
By Proposition \ref{prop:vandtoSel0}, $\vandpsiv{\psi^{-1}}{2-k}$ implies $\Sel(\Q,A_f) = 0$.  Thus, by Proposition \ref{prop:controlinf}, we have that $| \Sel(\Qinf,A_f)^{\Gamma} | \leq  q^{\ord_{\varpi}(a_p(f)-1)}$.
We note that by our choice of $T$, we have $H^0(\Q,A_f[\varpi])=0$ which is needed to invoke Proposition \ref{prop:controlinf}. Lastly,  by \cite[Proposition 2.5]{GV}, $\Sel(\Qinf,A_f)^\vee$ has no non-zero finite submodules; thus $\mu(\Sel(\Qinf,A_f)^\vee) \leq \ord_p(a_p(f)-1)$ by Lemma \ref{lemma:mu}.
\end{proof}

\begin{remark}
We note that if $j(\m)=2$ and $\psi(p) = 1$, then Proposition \ref{prop:vandtoSel0} does not hold.  Indeed, Lemma \ref{lemma:Wiles} gives in this case that the image of $H^1(\Q_\Sigma/\Q,\Ff(\psi\omega^{k-1}))$ in $H^1(\Q_\Sigma/\Q,A_f[\varpi])$ yields a non-trivial class in $\Sel(\Q,A_f[\varpi])$.  Nonetheless Theorem \ref{thm:algupperbndgen} likely holds in this case even if our method of proof fails.
\end{remark}

\subsection{Conjecture}

We continue with the notation and assumptions of the previous section so that $A_f = V_f / T_{\eis(f)}$.  Theorems \ref{thm:alglowerbndgen} and \ref{thm:algupperbndgen} give the following string of inequalities:
\begin{equation}
\label{eqn:algineq}
\eis(f) \leq \mu(\Sel(\Qinf,A_f)^\vee) \leq \ord_\varpi(a_p(f)-1).
\end{equation}

Greenberg has formulated precise conjectures on $\mu$-invariants of Selmer groups of an elliptic curve in terms of the Galois module structure of $E[p^n]$ for $n$ large enough \cite[Conjecture 1.11 and page 70]{Greenberg-CIME}.  These conjecture readily generalize to the case of modular forms and in this context predict that the lower bound gives the true value of the $\mu$-invariant. 

\begin{conj}
We have
$$
\mu \left( \Sel(\Qinf,A_f) ^\vee \right) = \eis(f).
$$
\end{conj}

\subsection{When the bounds meet}

As on the analytic side, when the upper and lower bound in (\ref{eqn:algineq}) meet, the Iwasawa theory becomes very simple.  In this section, we will prove the following  theorem which establishes that the Selmer group is entirely given by the $\mu$-invariant in this situation (compare with Theorem \ref{thm:anfullgen} on the analytic side).

\begin{thm}
\label{thm:selmer}
Let $\m \subseteq \T$ be an maximal ideal satisfying
\begin{itemize}
\item 
\fk
\item
\lu
\item 
$j(\m) = 2 \implies \psi(p) \neq 1$.  
\end{itemize} 
Then for $f$ a classical eigenform in the Hida family for $\mc$, we have  
$$
\Sel(\Qinf,A_f) \cong (\Lambda_{\O}/\varpi^{\eis(f)}\Lambda_{\O})^\vee.
$$
\end{thm}

\begin{proof}
To ease notation, let $S=\Sel(\Qinf,A_f)$.  First note that by Theorem \ref{thm:eisenprincipal} we know that $\T$ is Gorenstein as $\lu$ holds. Then by \cite[Theorem 1.2]{Wake-gorenstein}, we have that $\vandpsiv{\psi^{-1}}{2-j(\m)}$ holds. 
Hence, by Theorems \ref{thm:alglowerbndgen} and \ref{thm:algupperbndgen}, 
$$
\eis(f) \leq \mu(S^\vee) \leq \ord_\varpi(a_p(f)-1).
$$
Further, by Lemma \ref{lemma:Leis}, \lu~implies that $\eis(f) = \ord_\varpi (a_p(f)-1)$, and thus 
$\mu(S^\vee) = \ord_\varpi(a_p(f)-1)$.

Let $h \in \Lambda_{\O}$ denote the characteristic power series of $S^\vee$.  Then by \cite[Lemma 4.2]{Greenberg-CIME}, we have 
$$
\ord_q |S^\Gamma| = \ord_\varpi h(0) + \ord_q  |S_\Gamma|.
$$
Clearly,  $\ord_\varpi h(0) \geq \mu(S^\vee) = \ord_\varpi(a_p(f)-1)$.
Further, in the course of the proof of Proposition \ref{prop:controlinf}, the following exact sequence was derived:
\begin{equation}
\label{eqn:ct}
0 \to \Sel(\Q,A_f) \to S^\Gamma \to \ker(r)
\end{equation}
and that there was a surjection $\O /(a_p(f)-1)\O \surj \ker(r)$.  By Proposition \ref{prop:vandtoSel0}, we have $\Sel(\Q,A_f)=0$, and thus $\ord_q |S^\Gamma| \leq \ord_\varpi(a_p(f)-1)$.  Thus, we must have that 
$$
\ord_q |S^\Gamma| = \ord_\varpi h(0) = \ord_\varpi(a_p(f)-1) \text{~~and~~} \ord_q |S_\Gamma| = 0.
$$
But now $\mu(h) = \ord_\varpi h(0)$, and hence $\lambda(h) = 0$ (as in the proof of Theorem \ref{thm:anfullgen}). 

We now know that $\mu(S^\vee) = \eis(f)$ and $\lambda(S^\vee)=0$, and thus to prove our theorem it suffices to check that $S^\vee$ is a cyclic $\Lambda_\O$-module.  By Nakayama's lemma, it suffices to check that $(S^\vee)_{\Gamma}$ is a cyclic $\O$-module or equivalently that $S^{\Gamma}$ is a cyclic $\O$-module.   Returning to (\ref{eqn:ct}), we can now deduce that $S^\Gamma \cong \ker(r) \cong \O /(a_p(f)-1)\O$ as all of these modules have the same size.  Thus, $S^\Gamma$ is $\O$-cyclic which concludes the proof.
\end{proof}

\begin{cor}[Main conjecture]
\label{cor:mc}
Let $\m \subseteq \T$ be an maximal ideal satisfying
\begin{itemize}
\item 
\fk
\item
\lu
\item 
$j(\m) = 2 \implies \psi(p) \neq 1$.  
\end{itemize} 
Then for $f$ a classical eigenform in the Hida family for $\mc$, we have  
$$
\chr_{\Lambda_\O}( \Sel(\Qinf,A_f)^\vee ) = \varpi^{\eis(f)} \Lambda_{\O} = L_p(f/\Qinf) \Lambda_{\O}.
$$
\end{cor}

\begin{proof}
Simply combine Theorem \ref{thm:anfullgen} and Theorem \ref{thm:selmer}.
\end{proof}

\begin{appendix}
\label{appendix}

\section{Two-variable  $p$-adic $L$-functions}
\label{appendix:2var}

Let $\T$ denote the universal ordinary Hecke algebra of tame level $\Gamma_1(N)$ acting on modular forms (not just cuspforms), and let $\m$ denote a maximal ideal of $\T$.  We aim to construct a two-variable $p$-adic $L$-function over $\Tm$ assuming the maximal ideal $\m$ satisfies \eqref{m1}.  Specifically, we construct elements
$$
L_p^+(\m) \in \T_{\m} \otimes_{\Zp} \Zp[[\Zpx]] \text{~~and~~}
L_p^-(\m) \in \Tc_{\mc} \otimes_{\Zp} \Zp[[\Zpx]]
\footnote{Note that $\T_{\m} \cong \Tc_{\mc}$ if $\m$ is not Eisenstein and the necessary distinction being drawn between the plus and minus $p$-adic $L$-functions disappears.}
$$
such that for any classical height one prime $\p_f \subseteq \T_{\m}$ of residue characteristic zero, the image of $L_p^+(\m)$ in $\Tm/\p_f \otimes \Zp[[\Zpx]]$ yields the $p$-adic $L$-function $L^+_p(f)$, and similarly for $L_p^-(\m)$.

\subsection{Integral control theorems}

Recall the definition of  $\cP_g$ for $g \geq 0$ as well as the specialization map $\Mv{g} \to \cP_g^\vee$ from section \ref{sec:OMS}.  In this section, we aim to prove the following integral version of the control theorem in Theorem \ref{thm:control}.  In what follows, we write $H^i_c(X)$ for $H^i_c(\Gamma_0,X)$.

\begin{thm}
\label{thm:comparek}
Specialization induces an isomorphism
$$
H^1_c(\Mv{g})^{\ord} \stackrel{\sim}{\lra} H^1_c(\cP_g^\vee)^{\ord}.
$$
\end{thm}

Towards proving Theorem \ref{thm:comparek}, we recall the following famous result of Mahler.

\begin{thm}
\label{thm:Mahler}
Any continuous function $f :\Zp \to \Qp$ can be written uniquely in the form $f(x) = \sum_{j=0}^\infty c_j \binom{x}{j}$ with $\{c_j\}$ a sequence in $\Qp$ tending to 0.  Moreover, 
$$
\sup_{x \in \Zp} |f(x)|_p = \sup_j |c_j|_p.
$$
In particular, if $f$ takes values in $\Zp$, then all of the $c_j$ are integral.
\end{thm}

\begin{proof}
See \cite{Mahler}.
\end{proof}

\begin{cor}
\label{cor:Mahler}
The map which associates $\mu \in \M$ to the sequence $\{\mu ( \binom{x}{j} ) \}_{j=0}^\infty$ establishes a bijection between $\M$ and the collection of sequences in $\Zp$.  
\end{cor}

\begin{proof}
This corollary follows immediately from Theorem \ref{thm:Mahler}.
\end{proof}

\begin{lemma}
\label{lemma:specsurj}
The specialization map $\Mv{g} \to \cP_g^\vee$ is surjective.
\end{lemma}

\begin{proof}
By Theorem \ref{thm:Mahler}, $\cP_g$ is generated over $\Zp$ by the binomial coefficients $\binom{x}{j}$ for $0 \leq j \leq k$.  Thus this lemma follows directly from  Corollary \ref{cor:Mahler}.
\end{proof}

\begin{proof}[Proof of Theorem \ref{thm:comparek}]
By Lemma \ref{lemma:specsurj}, we have an exact sequence
$$
0 \to K \to \Mv{g} \to \cP_g^\vee \to 0
$$
where $K$ is the subspace of measures which vanish on polynomials of degree less than or equal to $k$.  The long exact sequence for cohomology then gives
$$
H^1_c(K)^{\ord} \to H^1_c(\Mv{g})^{\ord} \to H^1_c(\cP_g^\vee)^{\ord} \to H^2_c(K)^{\ord}.
$$

We first show $H^1_c(K)^{\ord}=0$.  To this end, take $\Phi \in H^1_c(K)^{\ord}$, and  for any $n$ we can write $\Phi = \Psi \big| U_p^n$ with $\Psi \in  H^1_c(K)^{\ord}$.  Then 
$$
\Phi(D) = (\Psi \big| U_p^n)(D) = \sum_{a=0}^{p-1} \Psi (\psmallmat{1}{a}{0}{p^n} D) \big| \psmallmat{1}{a}{0}{p^n}
$$
Since all values of $\Psi$ are in $K$, we have in particular that 
$$
\Psi(\psmallmat{1}{a}{0}{p^n} D)(\Zp) 
=
\Psi(\psmallmat{1}{a}{0}{p^n} D)({\bf 1}_{\Zp}) = 0.
$$
Thus, by Lemma \ref{lemma:measUp} below, we have $\Phi(D)(a+p^n\Zp) = 0$ for all $a$.  But since $n$ was chosen arbitrarily, we have that $\Phi(D)$ and thus $\Phi$ is identically zero.

To see $H^2_c(K)^{\ord}=0$, note that 
$$
H^2_c(K)^{\ord} \cong H_0(K)^{\ord} \cong (K_{\Gamma_0})^{\ord} 
\cong (K^{\ord})_{\Gamma_0}.
$$
Here we are computing $K^{\ord}$ with respect to the $U_p$-operator $\sum_{a=0}^{p-1} \psmallmat{1}{a}{0}{p}$.  It thus suffices to show that $K^{\ord} = 0$.  But this follows from Lemma \ref{lemma:Mord} below.
\end{proof}

\begin{lemma}
\label{lemma:measUp}
If $\mu \in \M$ such that $\mu(\Zp) = 0$, then 
$$
(\mu \big| \psmallmat{1}{a}{0}{p^n})(b + p^n \Zp) = 0
$$
for all $b \in \Zp$.
\end{lemma}

\begin{proof}
We have
\begin{align*}
(\mu \big| \psmallmat{1}{a}{0}{p^n})(b + p^n \Zp)
=
(\mu \big| \psmallmat{1}{a}{0}{p^n})({\bf 1}_{b + p^n \Zp}(z)) 
=
\mu({\bf 1}_{b + p^n \Zp}(a+ p^nz)) 
\end{align*}
When $a \not \equiv b \pmod{p^n}$, the function ${\bf 1}_{b + p^n \Zp}(a+ p^nz)$ is identically zero.  When $a \equiv b \pmod{p^n}$, we have 
$$
\mu({\bf 1}_{b + p^n \Zp}(a+ p^nz))  = \mu({\bf 1}_{\Zp}(z)) = \mu(\Zp) = 0
$$
by assumption.
\end{proof}

Define an action of $U_p$ on $\M$ simply by the action of $\sum_{a=0}^{p-1} \psmallmat{1}{a}{0}{p}$.

\begin{lemma}
\label{lemma:Mord}
We have $\M^{\ord} = 0$.
\end{lemma}

\begin{proof}
For $\mu \in \M^{\ord}$, write $\mu = \mu_n | U_p^n$ for each $n \geq 0$.  Then 
\begin{align*}
\mu(b+p^n\Zp) 
= \sum_{a=0}^{p^n-1} \left(\mu_n \big| \psmallmat{1}{a}{0}{p^n}\right)(b+p^n\Zp) = \mu_n(\Zp)
\end{align*}
which is independent of $b$.  But this clearly forces $\mu$ to vanish.
\end{proof}

We end this section with a ``mod $p$" control theorem. 

\begin{lemma}
\label{lemma:modp}
We have
$$
H^1_c(\cP_g^\vee)^{\ord} \otimes \Fp \cong H^1_c(\cP_g^\vee \otimes \Fp)^{\ord}.
$$
\end{lemma}

\begin{proof}
Starting with the short exact sequence
$$
0 \to \cP_g^\vee \stackrel{\times p}{\lra} \cP_g^\vee \lra \cP_g^\vee \otimes \Fp \to 0
$$
gives
$$
0 \to H^1_c(\cP_g^\vee)^{\ord} \stackrel{\times p}{\lra} H^1_c(\cP_g^\vee)^{\ord} \lra 
H^1_c(\cP_g^\vee \otimes \Fp)^{\ord} \to H^2_c(\cP_g^\vee)^{\ord}.
$$
As before, we have 
$$
H^2_c(\cP_g^\vee)^{\ord} \cong ((\cP_g^\vee)_{\Gamma_0})^{\ord} \cong
((\cP_g^\vee)^{\ord})_{\Gamma_0}.
$$
Since $\Mvx{g} \surj \cP_g^\vee$ (Lemma \ref{lemma:specsurj}) and $\Mvx{g}^{\ord} = 0$ (Lemma \ref{lemma:Mord}), we get $(\cP_g^\vee)^{\ord} = 0$ and thus $H^2_c(\cP_g^\vee)^{\ord}=0$.
\end{proof}

\subsection{Two-variable measures and specialization}

Recall the natural isomorphism
$$
\Ms \cong \Msg
$$
which sends the Dirac-measure $\delta_x$ supported at $x \in \X$ to the group-like element $[x]$ in $\Msg$.  We endow $\Msg$ with the $\Mxg$-action arising from the embedding
\begin{align*}
\Zpx &\to \Mxg \\
a &\mapsto a^2[(a,a)].
\end{align*}
The factor of $a^2$ here is related to the fact that weight $k$ modular forms correspond to $\Sym^{k-2}$-valued modular symbols. 
In terms of measures, the corresponding action of $\Mxg$ on $\Ms$ is given by
$$
([a] \cdot \mu)(f(x,y)) = a^2 \int_{\X} f(ax,ay) ~d\mu(x,y)
$$
for $a \in \Zpx$ and $\mu \in \Ms$.

The space $\Ms$ also admits an action of $\sigop$ defined by
$$
(\mu \big| \gamma)(f(x,y)) = \int_{\X} f((x,y) \cdot \gamma) ~d\mu(x,y).
$$
The specialization to weight $g$ is then the following $\sigop$-equivariant map:
\begin{align*}
\Ms &\lra \Mv{g} \\
\mu &\mapsto \left(f \mapsto \int_{\X} x^g f(y/x) ~d\mu(x,y) \right).
\end{align*}

Let $\p_k \subseteq \Mxg$ denote the ideal generated by elements of the form $[a] - a^k$.
We note that $\p_k$ is a principal ideal with generator $\pi_k = [\gamma]-\gamma^k$ where $\gamma$ is a topological generator of $\Zpx$.

\begin{lemma}
\label{lemma:specmeas}
The specialization map $\Ms \to \Mv{g}$ sits in the exact sequence
$$
0 \to \Ms \stackrel{\times \pi_{k}}{\lra} \Ms \to \Mv{k-2} \to 0
$$
where $\pi_{k}$ is any generator of the ideal $\p_{k}$.
\end{lemma}

\begin{proof}
To prove this lemma, we will work with the interpretation of these measure spaces as group algebras.  Specialization is then the map
$$
\Msg \to \Mg
$$
which sends the group-like element $[(a,b)]$ to $a^{k-2} [b/a]$.  We immediately see then that specialization is  surjective as $[(1,b)]$ maps to $[b]$.  

To see that multiplication by $\pi_k$ is injective, note that $\Zp[[\X]]$ is isomorphic to a direct sum of $p-1$ copies of $\Zp[[(1+p\Zp) \times \Zp]]$ with projections induced by the characters of $(\Z/p\Z)^\times$.  Since $\Zp[[(1+p\Zp) \times \Zp]]$ is a domain, it suffices to see that no projection of $\pi_k$ is 0 which is true by inspection.

Lastly, to compute the kernel of specialization, we will make a change of variables on $\Zp[[\X]]$.  Namely, consider the group isomorphism
\begin{align*}
\X &\to \X \\
(r,s) &\mapsto (r,rs)
\end{align*}
which induces a ring isomorphism
$$
\Msg \stackrel{\alpha}{\lra} \Msg. 
$$
The map $\alpha$ is $\Mxg$-linear if we endow the target with same $\Mxg$-action as before, but we endow the source the $\Mxg$-action arising from embedding
\begin{align*}
\Zpx &\to \Mxg \\
a &\mapsto a^2[(a,1)].
\end{align*}
By precomposing by $\alpha$, it  suffices to compute the kernel of $\Msg \to \Mg$ where the group-like element $[(a,b)]$ simply maps to $a^{k-2} [b]$.  This kernel is visibly equal to $\p_k \Msg$ as desired.
\end{proof}

\subsection{Two-variable control theorem}

To ease notation we write $\MM$ for $\Ms$ and $\MMv{g}$ for $\Mv{g}$.

\begin{thm}
\label{thm:control2to1}
Specialization induces the isomorphism
$$
H^1_c(\MM)^{\ord} \otimes_{\Zp} (\Mxg / \p_k)
 \cong
H^1_c(\MMv{k-2})^{\ord}.
$$
\end{thm}

\begin{proof}
Taking the corresponding long exact sequence of the exact sequence from Lemma \ref{lemma:specmeas}, and passing to ordinary parts then gives
\begin{align*}
0 \lra H^1_c(\MM)^{\ord} \stackrel{\times \pi_k}{\lra} H^1_c(\MM)^{\ord} \lra H^1_c(\MMv{k-2})^{\ord} \lra H^2_c(\MM)^{\ord}.
\end{align*}
We have
\begin{align*}
H^2_c(\MM)^{\ord} &\cong H_0(\MM)^{\ord} = (\MM_{\Gamma_0})^{\ord} 
= (\MM^{\ord})_{\Gamma_0}
\end{align*}
which one can see vanishes just as in the proof of the Lemma \ref{lemma:Mord}, completing the proof of the theorem.
\end{proof}

\begin{cor}
\label{cor:CT} 
For $k\geq 2$, we have
$$
H^1_c(\MM)^{\ord} \otimes (\Mxg / \p_k) \cong H^1_c(\cP_{k-2}^\vee)^{\ord}.
$$
\end{cor}

\begin{proof}
Combine Theorem \ref{thm:control2to1} and Theorem \ref{thm:comparek}.
\end{proof}

\subsection{Control theorems on the dual side}

For a maximal ideal $\m \subseteq \T$, the space $H^1_c(\cP_g^\vee \otimes \Fp)[\m_k]$ is more directly related to $\Hom_{\Lt}(H^1_c(\MM)_{\m},\Lt)$ than to $H^1_c(\MM)_{\m}$ where $\Lt := \Zp[[\Zpx]]$.  
For this reason, we now prove a control theorem on the dual side.  We begin with a few lemmas.

\begin{lemma}
\label{lemma:quot1}
Let $K$ be a field, $R$ a finitely generated $K$-algebra, and $M$ a finitely generated $R$-module.  For any ideal $I \subseteq R$, we have
\begin{enumerate}
\item 
$M[I]^\vee \cong M^\vee / I M^\vee$,
\item
$(M/IM)^\vee \cong M^\vee[I]$.
\end{enumerate}
Here $X^\vee = \Hom_K(X,K)$.
\end{lemma}

\begin{proof}
First note that the second isomorphism implies the first.  Indeed, applying the second isomorphism to $M^\vee$ gives
$$
(M^\vee / I M^\vee)^\vee \cong (M^\vee)^\vee[I]
$$
and thus
$$
(M^\vee / I M^\vee) \cong M[I]^\vee.
$$
To see the second isomorphism, there is a natural map
$$
(M / IM)^\vee \to M^\vee
$$
which is clearly injective.  Moreover, the image of this map lands in $M^\vee[I]$ since if $a \in I$ and $\phi : M \to M/IM \to K$ is in the image, then $(a \cdot \phi)(m) = \phi(am) = 0$ for all $m \in M$.

Lastly, to get surjectivity, take $\phi \in M^\vee[I]$.  So $\phi: M \to K$ and $a \cdot \phi = 0$ for all $a \in I$.  To prove surjectivity, we simply need to see that $\phi(IM) = 0$.  But this follows since for $a \in I$, $\phi(am) = (a \cdot \phi)(m) = 0$ 
\end{proof}

\begin{lemma}
\label{lemma:quot2}
Let $M$ be a projective $R$-module and $I=aR$ a principal ideal of $R$ where $a$ is not a zero-divisor.  Then
$$
\Hom_R(M,R) / I \Hom_R(M,R) \cong \Hom_R(M,R/I) \cong \Hom_{R/I}(M/IM,R/I).
$$
\end{lemma}

\begin{proof}
Since $a$ is not a zero-divisor, we have
$$
0 \lra R \stackrel{\times a}{\lra} R \lra R/I \lra 0,
$$
and thus
$$
0 \lra \Hom_R(M,R) \stackrel{\times a}{\lra} \Hom_R(M,R) \lra
\Hom_R(M,R/I) \lra \Ext^1_R(M,R).
$$
This last Ext group vanishes as $M$ is projective over $R$ which proves the lemma.
\end{proof}

The following lemma gives control theorems for $\Hom_{\Lt}(H^1_c(\MM)_{\m}, \Lt)$.

\begin{lemma}
\label{lemma:lemma}
For $k \equiv j(\m) \pmod{p-1}$ and $g=k-2 \geq 0$, we have
$$
\Hom_{\tilde{\Lambda}}(H^1_c(\MM)_{\m},\tilde{\Lambda}) \otimes_{\tilde{\Lambda}} \tilde{\Lambda}/\p_k \cong \Hom(H^1_c(\cP_g^\vee)_{\m_k} ,\Zp)
$$
and
$$
\Hom_{\tilde{\Lambda}}(H^1_c(\MM)_{\m},\tilde{\Lambda}) \otimes_{\T} \T/\m \cong \Hom(H^1_c(\cP_g^\vee \otimes \Fp)[\overline{\m}_k] ,\Fp).
$$
\end{lemma}

\begin{proof}
Set $X = H^1_c(\MM)^{\ord}$.  Then for the first part we have
\begin{align*}
\Hom_{\tilde{\Lambda}}(X_{\m},\tilde{\Lambda}) \otimes_{\tilde{\Lambda}} \tilde{\Lambda}/\p_k
&\cong
\Hom_{\tilde{\Lambda}}(X_{\m}/\p_k X_{\m},\tilde{\Lambda}/\p_k) & (\text{Lemma~} \ref{lemma:quot2})\\
&\cong
\Hom_{\tilde{\Lambda}}((X/\p_kX)_{\m_k},\tilde{\Lambda}/\p_k) & \\
&\cong
\Hom(H^1_c(\cP_g^\vee)_{\m_k} ,\Zp). & (\text{Corollary~} \ref{cor:CT})
\end{align*}
For the second part, we have
\begin{align*}
\Hom_{\tilde{\Lambda}}(X_{\m},\tilde{\Lambda}) \otimes_\T \T/\m 
&\cong
\Hom_{\Lt}(X_{\m},\Lt) \otimes_{\Lt} (\Lt/\p_k) \otimes_{\T} \T/\m \\
&\cong
\Hom(H^1_c(\cP_g^\vee)_{\m_k},\Zp)  \otimes_{\T} \T/\m  & (\text{part~} 1 \text{~above})\\
&\cong
\Hom(H^1_c(\cP_g^\vee)_{\m_k}\otimes_{\Zp} \Fp,\Fp) \otimes_{\T} \T/\m  & (\text{Lemma~} \ref{lemma:quot2})\\
&\cong
\Hom(H^1_c( \cP_g^\vee \otimes \Fp)_{\m_k},\Fp)\otimes_{\T} \T/\m & (\text{Lemma~} \ref{lemma:modp})\\
&\cong
\Hom(H^1_c( \cP_g^\vee \otimes \Fp)[\m_k],\Fp). & (\text{Lemma~} \ref{lemma:quot1})
\end{align*}
\end{proof}

\subsection{Freeness over the Hecke-algebra}

For $\m \subseteq \T$ maximal,
we aim to show that $\m$ satisfies a mod $p$ multiplicity one assumption if and only if $\Hom_{\Lt}(H^1_c(\MM)^{\ve}_{\m}, \Lt)$ is free over $\Tm$ for $\ve = +$ and free over $\Tcm$ if $\ve = -$.  However, such a result cannot hold unconditionally as the Eisenstein Hecke-eigensystems do not always all occur in the plus subspace.  We thus introduce the following hypothesis on $\m$ to force this condition to hold.

\begin{defn}
We say $\m \subset \T$ satisfies (\ep) if $\Hom_{\Gamma_0}(\Delta,\cP_{k-2}^\vee)_{\mk}^-=0$ for all (equivalently for one) classical $k \equiv j(\m) \pmod{p-1}$.
\end{defn}

Note that (\ep) is automatically satisfied if $\m$ is not an Eisenstein maximal ideal.  It is also automatically satisfied if $N$ is squarefree.  But it fails for instance when $N=9$ as the boundary symbol attached to $E_{k,\chi,\chi}$ lies in the minus subspace where $\chi$ is the non-trivial (odd) character of order 3.

\begin{lemma}
\label{lemma:ES}
If $\m \subseteq \T$ satisfies (\ep), for $k\geq2$, $k \equiv j(\m) \pmod{p-1}$, we have
$$
H^1_c(\cP_{k-2}^\vee)_{\mk}^+ \otimes \Qpbar \cong M_{k}(\Gamma_0,\Zp)_{\mk} \otimes \Qpbar
$$
and
$$
H^1_c(\cP_{k-2}^\vee)^-_{\mk} \otimes \Qpbar \cong S_{k}(\Gamma_0,\Zp)_{\mk} \otimes \Qpbar
$$
as Hecke-modules.
\end{lemma}

\begin{proof}
The lemma follows immediately from \cite[Proposition 3.15]{B-critical}.
\end{proof}

\begin{thm}
\label{thm:m1impliesfreeplus}
Let $\m$ be a maximal ideal of $\T$ satisfying (\ep).  The following are equivalent:
\begin{enumerate}
\item \label{1}
$\Hom_{\Lt}(H^1_c(\MM)^+_{\m}, \Lt),$ is a free $\Tm$-module of rank 1;
\vspace{.1cm}
\item \label{2}
$\dim_{\T/\m} H^1_c(\cP_{k-2}^\vee \otimes \Fp)^{+}[\om_k] =1$ for all $k\geq 2$, $k \equiv j(\m) \pmod{p-1}$;
\item \label{3}
$\dim_{\T/\m} H^1_c(\cP_{k-2}^\vee \otimes \Fp)^{+}[\om_k] =1$ for some $k\geq 2$, $k \equiv j(\m) \pmod{p-1}$.
\end{enumerate}
\end{thm}

\begin{proof}
Lemma \ref{lemma:lemma} gives (\ref{1}) implies (\ref{2}).  The implication (\ref{2}) implies (\ref{3}) is clear.  So we just need to show that (\ref{3}) implies (\ref{1}).  
To this end,
by Lemma \ref{lemma:lemma} and Nakayama's lemma, (\ref{3}) implies that $Y^+:=\Hom_{\Lt}(H^1_c(\MM)^+_{\m}, \Lt)$ is a cyclic $\Tm$-module.  We then have an exact sequence
$$
0 \to K \to \T \to Y^+ \to 0
$$
and thus
$$
0 \to K/\p_kK \to \T/\p_k\T \to Y^+/\p_kY^+ \to 0
$$
as $Y^+$ is $\tilde{\Lambda}$-torsion free.  Note that if we show for a single $k$ that $K/\p_kK = 0$, then, by Nakayama's lemma, $K=0$, and $Y^+$ is free of rank 1 over $\T$.  

To this end, recall that $\T/\p_k\T \cong \Tk$ which is a torsion-free $\Zp$-module.  Also, $Y^+/\p_kY^+ \cong 
\Hom(H^1_c(\cP_{k-2}^\vee)^+_{\m_k} ,\Zp)$ by Lemma \ref{lemma:lemma} which is also torsion-free.  Thus, it suffices to see that the $\Zp$-ranks of $\Tkm$ and $H^1_c(\cP_{k-2}^\vee)^+_{\m_k}$ match.  But this follows immediately from Lemma \ref{lemma:ES}.
\end{proof}

\begin{thm}
\label{thm:m1impliesfreeminus}
Let $\mc$ be a maximal ideal of $\Tc$ satisfying (\ep).  The following are equivalent:
\begin{enumerate}
\item 
$\Hom_{\Lt}(H^1_c(\MM)^-_{\mc}, \Lt),$ is a free $\Tcm$-module of rank 1;
\vspace{.1cm}
\item 
$\dim_{\T/\m} H^1_c(\cP_{k-2}^\vee \otimes \Fp)^{-}[\omk] =1$ for all $k \geq 2$, $k \equiv j(\m) \pmod{p-1}$;
\item 
$\dim_{\T/\m} H^1_c(\cP_{k-2}^\vee \otimes \Fp)^{-}[\omk] =1$ for some $k \geq 2$, $k \equiv j(\m) \pmod{p-1}$.
\end{enumerate}
\end{thm}

\begin{proof}
We first need to justify why $X^-:=H^1_c(\MM)^-_{\m}$ is even a module over $\Tcm$.  To this end, take $h$ in the kernel of $\Tm \to \Tcm$, and we will show that $h$ kills $X^-$.  For $x \in X^-$, consider the image of $hx$ in 
$$
X^- \otimes_{\tilde{\Lambda}} \tilde{\Lambda}/\p_k \cong H^1_c(\cP_{k-2}^\vee)^-_{\mk}.
$$
Further, by Lemma \ref{lemma:ES}, we have $H^1_c(\cP_{k-2}^\vee)_{\mk}^- \inj S_k(\Gamma_0,\Qpbar)$.  Thus, the image of $hx$ in $X^-/\p_k X^-$ vanishes.  Since this is true for all $k \geq 2$, we deduce that $hx =0$ and that $X^-$ is a $\Tcm$-module.  The remainder of the proof follows exactly as in the proof of Theorem \ref{thm:m1impliesfreeplus}.
\end{proof}

\subsection{Two-variable $p$-adic $L$-functions}
\label{sec:twovar}

In this section, we assume that 
$$
\dim_{\T/\m} H^1_c(\cP_{k-2}^\vee \otimes \Fp)^{\ve}[\om_k] =1
$$
so that we know that $\Hom_{\Lt}(H^1_c(\MM)_{\m}^{\ve},\Lt)$ is a free $\T^{\ve}_{\m}$-module where $\T^+_{\m} := \Tm$ and $\T^-_{\m} := \Tcm$.  Moreover, we now fix an isomorphism $\Hom_{\Lt}(H^1_c(\MM)_{\m}^{\ve},\Lt) \cong \T^{\ve}_{\m}$.

Consider the $\Lt$-linear map
\begin{align*}
 H^1_c(\MM)^{\ord}   &\to \MM := \Ms. \\
\Phi &\mapsto \Phi(\{\infty\} - \{0\})
\end{align*}
Let $L_p^\ve$ denote the corresponding element of
\begin{align*}
\Hom_{\Lt}(H^1_c(\MM)^{\ord,\ve},\MM) 
&\cong \Hom_{\Lt}(H^1_c(\MM)^{\ord,\ve},\Lt) \hat{\otimes}_{\Lt} \Zp[[\Zpx \times \Zp]] \\
&\cong \Hom_{\Lt}(H^1_c(\MM)^{\ord,\ve},\Lt) \hat{\otimes}_{\Lt} \left(\Lt \hat{\otimes}_{\Zp} \Zp[[\Zp]] \right)\\
&\cong \Hom_{\Lt}(H^1_c(\MM)^{\ord,\ve},\Lt) \hat{\otimes}_{\Zp} \Zp[[\Zp]].
\end{align*}
Here the second isomorphism arises from
\begin{align*}
\Zp[[\Zpx \times \Zp]] &\cong \Lt \hat{\otimes}_{\Zp} \Zp[[\Zp]]\\
[(a,b)] &\mapsto [a] \otimes [b/a]
\end{align*}
where $\Lt$ acts on $\Lt \hat{\otimes}_{\Zp} \Zp[[\Zp]]$ simply by acting on the first coordinate.

Let $L^\ve_p(\m)$ denote the image of $L_p^\ve$ in 
$$
\Hom_{\Lt}(H^1_c(\MM)^{\ve}_{\m},\MM) \cong \Hom_{\Lt}(H^1_c(\MM)^{\ve}_{\m},\Lt) \hat{\otimes}_{\Zp} \Zp[[\Zp]] 
\cong \T_{\m}^{\ve} \hat{\otimes}_{\Zp} \Zp[[\Zp]] \cong \Tm^{\ve}[[\Zp]].
$$
For $\p$ a height 1 prime ideal of $\Tm^{\ve}$ with $\O := \Tm^{\ve}/\p$, we write  $\p(L^\ve_p(\m))$ for the image of $L^\ve_p(\m)$ under the map
$$
\Tm^{\ve}[[\Zp]] \surj (\Tm^{\ve} / \p)[[\Zp]] \cong \O[[\Zp]].
$$
We call $L_p^\ve(\m)$ the two-variable $p$-adic $L$-function attached to $\m$ as it has the following interpolation property.  

\begin{prop}
\label{prop:interpolate}
Let $\m$ be a maximal ideal of $\T$, and let $\p_f$ be a classical height one prime of $\T$ corresponding to an eigenform $f$.
Let $\O$ denote the subring of $\Qpbar$ generated by the Hecke-eigenvalues of $f$.  Then
$$
\p_f( L_p^\ve(\m) ) \Big|_{\Zpx} = L_p^\ve(f)
$$
as $\O$-valued measures on $\Zpx$ (for some choice of canonical period $\Omega_f^\ve$).
\end{prop}

\begin{proof}
Given the control theorems proven in this appendix, the proof of this theorem is just a diagram chase.
\end{proof}

\end{appendix}

\bibliography{mu}
\bibliographystyle{abbrv}

\end{document}